\documentclass[11pt]{amsart}

\usepackage{graphicx}
\usepackage[english]{babel}
\usepackage[T1]{fontenc}
\usepackage[latin1]{inputenc}
\usepackage{amsfonts}
\usepackage{amssymb}
\usepackage{amsthm}
\usepackage{amsmath}
\usepackage{tikz}
\usepackage{float}
\usepackage{enumerate}
\usepackage{accents}
\usepackage{mathtools}
\usepackage{mathrsfs}
\usepackage{comment}
\usepackage{afterpage}
\usepackage{bbm}
\usepackage{dsfont}
\usepackage{stmaryrd}
\usepackage{hyperref}
\usepackage{mleftright}
\usepackage{subfig}
\usepackage{soul}
\usepackage{tikz-cd} 
\usepackage{overpic}

\theoremstyle{plain}
\newtheorem{thm}{Theorem}[section]
\newtheorem{lem}[thm]{Lemma}
\newtheorem{prop}[thm]{Proposition}
\newtheorem{cor}[thm]{Corollary}
\newtheorem*{thm*}{Theorem}
\newtheorem*{prop*}{Proposition}
\newtheorem*{cor*}{Corollary}
\newtheorem{thmintro}{Theorem}

\newtheorem{propintro}[thmintro]{Proposition}

\theoremstyle{definition}
\newtheorem{defn}[thm]{Definition}
\newtheorem{ex}[thm]{Example}
\newtheorem{rmk}[thm]{Remark}
\newtheorem*{rmk*}{Remark}
\newtheorem{conj}{Conjecture}
\newtheorem{quest}[conj]{Question}
\newtheorem*{quest*}{Question}

\newtheorem*{defn*}{Definition}

\renewcommand{\o}{\circ}

\newcommand{\R}{\mathbb{R}}
\newcommand{\Z}{\mathbb{Z}}
\newcommand{\N}{\mathbb{N}}

\newcommand{\s}{\sigma}
\newcommand{\ra}{\rightarrow}

\newcommand{\cu}{\subseteq}
\newcommand{\g}{\gamma}
\newcommand{\G}{\Gamma}

\newcommand{\mc}{\mathcal}
\newcommand{\mf}{\mathfrak}
\newcommand{\x}{\times}
\newcommand{\eps}{\epsilon}
\newcommand{\Om}{\Omega}

\newcommand{\Aut}{\mathrm{Aut}}

\newcommand{\acts}{\curvearrowright}

\newcommand{\mscr}{\mathscr}

\newcommand{\CAT}{{\rm CAT(0)}}

\newcommand{\Out}{\mathrm{Out}}

\DeclareMathOperator{\diam}{diam} 
\DeclareMathOperator{\supp}{supp}
\DeclareMathOperator{\comp}{comp}
\renewcommand{\L}{\Lambda}
\newcommand{\parclass}{\mathrm{Para}}
\newcommand{\cmp}{\mathrm{Cmp}}
\newcommand{\trans}{\mathrm{Trans}}

\begin{document}

\title{Deforming cubulations of hyperbolic groups}
\author[E. Fioravanti]{Elia Fioravanti}\address{Max Planck Institute for Mathematics, Bonn, Germany}\email{fioravanti@mpim-bonn.mpg.de} 
\author[M. Hagen]{Mark Hagen}\thanks{Hagen was supported by EPSRC grant EP/R042187/1.}\address{School of Mathematics, 
University of Bristol, Bristol, UK}
\email{markfhagen@posteo.net}

\begin{abstract}
We describe a procedure to deform cubulations of hyperbolic groups by ``bending hyperplanes''. Our construction 
is inspired by related constructions like Thurston's Mickey Mouse example \cite[Example~8.7.3]{Thurston}, walls in fibred 
hyperbolic $3$--manifolds~\cite{Dufour} and free-by-$\mathbb Z$ groups~\cite{Hagen-Wise2}, 
and Hsu--Wise \emph{turns} \cite[Definition~4.1]{Hsu-Wise}.

As an application, we show that every cocompactly cubulated Gromov-hyperbolic group admits a proper, cocompact, essential 
action on a $\CAT$ cube complex with a single orbit of hyperplanes.  This answers (in the negative) a question of Wise 
\cite[Problem~5.2]{Wise-antenna}, who proved the result in the case of free groups. 

We also study those cubulations of a general group $G$ that are not susceptible to trivial deformations.  We name these 
\emph{bald cubulations} and observe that every cocompactly cubulated group admits at least one bald cubulation. We then 
apply the hyperplane-bending construction to prove that every cocompactly cubulated hyperbolic group $G$ admits infinitely 
many 
bald cubulations, provided $G$ is not a virtually free group with $\Out(G)$ finite. By contrast, we show that the Burger--Mozes examples each admit a unique bald cubulation.
\end{abstract}

\maketitle

\setcounter{tocdepth}{1}
\tableofcontents

\section{Introduction.}

The theory of group actions on $\CAT$ cube complexes, and in particular its applications to 
$3$--manifolds~\cite{Agol,Wise-qch}, has recently exerted a large influence in group theory and topology.  ``Cubulating'' a 
group --- constructing a proper action on a $\CAT$ cube complex, usually via the method introduced by Sageev in~\cite{Sag95} 
--- reveals a great deal about the group's structure.  This is particularly true when the group $G$ is hyperbolic: in this 
case, when the \emph{codimension--$1$ subgroups} used to cubulate the group are quasiconvex, the action is 
also cocompact~\cite{Sag97,Niblo-Reeves,Hruska-Wise-finiteness}; we say $G$ is \emph{cocompactly cubulated} if it acts 
properly and cocompactly on a $\CAT$ cube complex.  In this case, work of Agol~\cite{Agol} and 
Haglund--Wise~\cite{Haglund-Wise-GAFA} shows that such a hyperbolic group $G$ has many useful properties, e.g.\ 
$\Z$--linearity, separability of quasiconvex subgroups, etc.

Many of the procedures for cubulating hyperbolic groups $G$ arising in nature make it clear that cubulations of $G$ are 
non-canonical and proving their existence is often non-constructive.  Proofs that a given $G$ is cubulated often proceed as 
follows.  First, one describes a general procedure for finding quasiconvex codimension--$1$ subgroups in $G$.  Then, one 
shows that any two points in the Gromov boundary of $G$ can be separated by the limit set of some coset of a codimension--$1$ 
subgroup of the given type: one constructs a particular subgroup ``targeted'' at the given pair of boundary points.  Then one 
applies a theorem of Bergeron--Wise~\cite{Bergeron-Wise}, relying on a compactness argument, to extract a finite collection 
of codimension--$1$ subgroups that suffice to ensure a proper action on a $\CAT$ cube complex.  

For example, when $G$ is the fundamental group of a hyperbolic $3$--manifold $M$, 
the codimension--$1$ subgroups can be taken to be fundamental groups
of quasi-Fuchsian surfaces immersed in $M$. The work of Kahn and Markovic \cite{Kahn-Markovic} 
shows that these are enough to separate any two points in the boundary of $G$.

While a lot of information about $G$ can be gleaned from the mere fact of it being cubulated, one usually does not know 
much about the specific cube complex.  It is therefore natural to want some sort of ``space of all cocompact cubulations'' 
of a given hyperbolic group $G$.

One way to proceed is by analogy to deformation spaces of actions on trees, introduced by Forester in~\cite{Forester}.  
There, one considers all \emph{minimal} actions of $G$ on trees in which the set of elliptic subgroups is held fixed.  In our setting, 
one might wish to consider all of the proper, cocompact actions of $G$ on $\CAT$ cube complexes.  The right notion of a ``minimal'' 
action on a $\CAT$ cube complex $X$ should at least include the requirement that there is no $G$--invariant convex 
subcomplex, so one should restrict to actions that are \emph{essential} in the sense of Caprace--Sageev~\cite{CS}; this avoids 
distractions like attaching a leaf edge to each vertex, or taking the product of $X$ with a finite cube complex.  A result 
in~\cite{CS} makes this a safe restriction, since any cocompact cubulation can be replaced with an essential one without 
really changing much.

However, one should also impose some additional restrictions that are best illustrated by considering the simple case where 
$G=\Z$.  The most obvious cubulation, the action by translations on the tiling of $\R$ by $1$--cubes, seems intuitively 
better than the action on the cube complex obtained by, say, stringing together countably many squares, with each 
intersecting the next in a single vertex.  Both of these cubulations are essential, and in both cases the 
hyperplane-stabilisers are trivial, but only in the first case is the action of $\{1\}$ on each hyperplane the ``right'' 
cubulation of the trivial group.

So, we ask that $X$ is \emph{hyperplane-essential}: every hyperplane-stabiliser acts essentially on its hyperplane.  This, 
too, turns out to be reasonable, in the sense that, if $G$ admits a proper, cocompact action on a $\CAT$ cube complex, then 
it admits a proper, cocompact, hyperplane-essential action~\cite{Hagen-Touikan}.  Passing to a hyperplane-essential action 
is somewhat more violent than making an action essential, since it seriously changes which subgroups are 
hyperplane-stabilisers.

\begin{rmk*}[Hyperplane-essentiality]
Essentiality and hyperplane-essentiality are defined precisely in Section~\ref{sec:prelim}. When $G$ is the fundamental 
group of a hyperbolic $3$--manifold, the cubulations provided by using Kahn--Markovic surfaces are automatically essential 
and hyperplane-essential (see Remark~\ref{automatic hyperplane-essentiality}).  

There are also stronger conditions that one 
might want to impose on a cubulation $G\acts X$.  For example, $X$ has \emph{codimension--$k$} hyperplanes, each of which is 
the intersection of $k$ pairwise-transverse hyperplanes; one could ask that the stabiliser of each hyperplane of each 
codimension acts essentially on it.  This condition is strictly stronger than hyperplane-essentiality (which imposes restrictions 
only on codimension--$1$ hyperplanes) and is equivalent to $X$ having no free faces; equivalently, the CAT(0) metric has the 
\emph{geodesic extension property}~\cite[Proposition II.5.10]{BH}.  However, it is not at all clear whether one can pass from an 
arbitrary cocompact cubulation to a cubulation with no free faces, so we work with hyperplane-essential actions instead of 
the stronger version.
\end{rmk*}

Restricting to proper, cocompact, essential, hyperplane-essential actions seems to be reasonable for the purpose of 
considering the ``space of cubulations'' of $G$ because of the following theorem of Beyrer and the first author~\cite{BF2}. 
For hyperbolic groups $G$ acting on cube complexes $X$ with the above properties, the action $G\acts X$ is completely 
determined, up to $G$--equivariant cubical isomorphism, by the \emph{length function} $\ell_X\colon G\to\N$.  
This is the function $\ell_X(g)=\inf_{x\in X}d(x,gx)$, where $d$ is the $\ell_1$ 
metric on $X$.  Throughout this paper, we will say that cubulations $G\acts X,G\acts Y$ are \emph{equivalent} 
if there exists a $G$--equivariant cubical isomorphism $X\to Y$.

This suggests a natural topology for the space of 
such cubulations~\cite{BF1}.  First, we allow ourselves to continuously vary the $\ell_1$ metric on $X$ by replacing the 
cubes by \emph{cuboids} --- the side-lengths need not be $1$, and we can continuously vary length functions by rescaling 
edges (always assigning the same length to parallel edges).  From this point of view, passing from $X$ to a cubical 
subdivision --- equivariantly replacing each hyperplane with several parallel copies but keeping the metric fixed --- has no effect 
on the length function, which now takes values in $\R_{\ge0}$.  Regarding each geometric, essential, 
hyperplane-essential action on a $\CAT$ cuboid complex as a length function gives a map from the set of such cubulations of 
$G$ to $\R^G-\{0\}$, which we equip with the product topology.  One can also projectivise, regarding as equivalent 
any two cubulations inducing homothetic length functions.  

This suggests that we should not consider cubulations $G\acts X,G\acts Y$ essentially different if they admit a common 
subdivision.  This motivates us to consider only cubulations in which no two halfspaces are at finite Hausdorff distance
(see the notion of ``bald'' cubulation in Definition~\ref{defn:bald} below).

In this paper, we concern ourselves with deformations of a cubulation $G\acts X$, i.e.\ with moving around in the space of 
cubulations.  We leave discussion of the subject from the point of view of the above topology for later work.  
Instead, we are concerned with a much more basic question: 
\begin{quest*}
\emph{For which (hyperbolic) cocompactly cubulated groups $G$ is the space of 
essential, hyperplane-essential cubulations infinite, even up to subdivision?}  
\end{quest*}

One way to move in the space of cubulations is using the action of $\Out(G)$; this varies the $G$--action, but not the 
underlying cube complex.  It is not hard to show that, if $\Out(G)$ is infinite, then the existence of a cocompact 
cubulation of $G$ ensures that there are infinitely many with the above properties. However, a common 
situation 
is where $G$ is a one-ended hyperbolic group that does not admit any two-ended splittings, so $\Out(G)$ 
is finite.  

In general, one needs to vary the cube complex as well as the action.  In this paper, we \emph{bend hyperplanes} to 
transform one cubulation into another and answer the above question.

\subsection*{Bending hyperplanes}
Let $G$ be a one-ended hyperbolic group acting properly and cocompactly on an essential, hyperplane-essential $\CAT$ cube 
complex $X$.  We now describe the \emph{bending} procedure for deforming $G\acts X$ into a new cubulation.  

The idea 
is to produce, given the hyperplanes of $X$, a new \emph{crooked hyperplane} $\mc{C}$ built from pieces of old hyperplanes. 
Each piece is obtained from some hyperplane $\mf{u}\subseteq X$ by cutting $\mf{u}$ along a family $\{\mf{w}_i\}$ of pairwise-disjoint hyperplanes 
that intersect it. Each $\mf{w}_i$ is itself cut along a family of hyperplanes containing $\mf{u}$, and also contributes a piece to the crooked hyperplane.
Finally, the various pieces are glued along their boundaries, which are codimension--$2$ hyperplanes in $X$. This is depicted in Figure~\ref{fig:crooked}.

\begin{figure}
\begin{tikzpicture}
\draw [line width=0.025cm] (-2.5,0) -- (-0.5,0);
\draw [dashed] (-1.5,-1) -- (-1.5,1);
\draw [-Latex] (-0.25,0) -- (0.25,0);
\draw [dashed] (1.5,0) -- (0.5,0);
\draw [dashed] (1.5,-1) -- (1.5,0);
\draw [line width=0.05cm,blue] (2.5,0) -- (1.5,0);
\draw [line width=0.05cm,blue] (1.5,0) -- (1.5,1);
\node [above right] at (-1.5,-1) {$\mf{w}_i$}; 
\node [above right] at (-2.5,0) {$\mf{u}$};
\node [above right] at (1.5,-1) {$\mf{w}_i$}; 
\node [above right] at (0.5,0) {$\mf{u}$};
\end{tikzpicture}
\begin{tikzpicture}
\draw [opacity=0] (-2.5,0) -- (-0.5,0);
\draw [-Latex] (-0.25,0) -- (0.25,0);
\draw [-Latex] (-2.25,0) -- (-1.75,0);
\draw [line width=0.05cm,blue] (1,-.5) -- (2,-.5);
\draw [line width=0.05cm,blue] (1,-.5) -- (.5,.5);
\draw [line width=0.05cm,blue] (2,-.5) --  (2.5,.5);
\draw [line width=0.05cm,blue] (.5,.5) -- (.75,1);
\draw [line width=0.05cm,blue] (2.5,.5) -- (2.25,1);
\draw [dashed] (.5,-.5) -- (2.5,-.5);
\draw [dashed] (.25,1) -- (1.25,-1);
\draw [dashed] (1.75,-1) --  (2.75,1);
\node [right] at (2.25,0) {\textcolor{blue}{\pmb{$\mc{C}$}}};
\node [] at (-1,0) {$\dots$};
\node [below left] at (2.75,-.5) {$\mf{u}$};
\node [below left] at (.35,1) {$\mf{w}_i$};
\node [below right] at (2.65,1) {$\mf{w}_j$};
\end{tikzpicture}
\caption{Bending hyperplanes into a crooked hyperplane.}
\label{fig:crooked}
\end{figure}
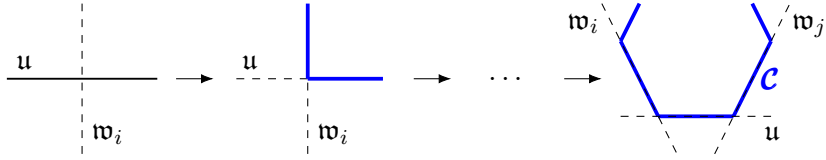

Since $G$ is hyperbolic, results of Agol~\cite{Agol} and Haglund--Wise~\cite{Haglund-Wise-GAFA} imply that 
hyperplane-stabilisers are separable.  This allows one to choose the pieces so that their bounding codimension--$2$ 
hyperplanes are all far apart, which in turn allows one to produce a crooked hyperplane that is $2$--sided and quasiconvex 
in $X$. If the pieces are constructed equivariantly with respect to a finite-index subgroup of $G$, then the crooked hyperplane is 
also acted upon cocompactly by its stabiliser.  Hence each crooked hyperplane corresponds to a 
quasiconvex codimension--$1$ subgroup of $G$, along with a specified wall; applying Sageev's construction~\cite{Sag95,Sag97} 
yields a cocompact, essential $G$--action on a new cube complex $Y$ with a single orbit of hyperplanes.  

With a bit more 
care, we can do the bending in such a way that every infinite-order $g\in G$ has its axis cut by some translate of the 
crooked hyperplane, ensuring, by~\cite{Bergeron-Wise}, that the $G$--action on $Y$ is proper.

As an application of the bending procedure, we can therefore answer a question asked by Wise in~\cite[Problem 
5.2]{Wise-antenna}:

\begin{thmintro}\label{one orbit intro}
Let $G$ be a Gromov-hyperbolic group that admits a proper, cocompact action on a $\CAT$ cube complex.  Then there exists a 
$\CAT$ cube complex $X$ on which $G$ acts properly, cocompactly, and essentially with a single orbit of hyperplanes.
\end{thmintro}

The above theorem was proved by Wise in the case where $G$ is a free group. He used an ingenious \emph{antenna} 
construction to produce a codimension--$1$ subgroup $H$ of $G$, and an associated $H$--wall, so that the 
action on the resulting dual cube complex has the claimed properties. In fact, Wise goes 
considerably further in the free group case: his construction shows that one can choose $H$ to be an \emph{arbitrary} 
finitely generated infinite-index subgroup.

As described above, our proof proceeds along completely different lines in the one-ended case, relying on bending hyperplanes. 
When $G$ is a surface group, the resulting cubulation essentially originates from a single (necessarily non-simple) filling closed curve on the surface.

In the general 
case, we split $G$ as a finite graph of groups with finite edge groups and use a hybrid technique: we apply a version of 
Wise's antenna construction to the Bass--Serre tree, apply the bending construction to the various one-ended vertex groups, 
and glue up the pieces to get the required wall.

\begin{rmk*}
In the one-ended case, our proof of Theorem~\ref{one orbit intro} actually shows more. 
Let $K\leq G$ be a hyperplane-stabiliser in a proper, cocompact action of $G$ on an essential, hyperplane-essential 
$\CAT$ cube complex. Then, for every open neighbourhood $U\cu\partial_{\infty}G$ of the limit set of $K$, the cubulation in Theorem~\ref{one orbit intro} can be chosen to have a hyperplane-stabiliser $H$ with limit set contained in $U$. The two sides of the $H$--wall can similarly be picked arbitrarily close to the two sides of the $K$--wall. 

In other words, an arbitrarily small deformation of one of the original walls always suffices to obtain a proper action
with a single orbit of hyperplanes.
\end{rmk*}

\subsection*{Bald cubulations}
Let $G$ be a (not necessarily hyperbolic) group acting properly and cocompactly on a $\CAT$ cube complex $X$.  It is not difficult to produce infinitely 
many cocompact cubulations of $G$, no two of which are equivalent.  This is because of various relatively uninteresting 
procedures:
\begin{itemize}
     \item we can cubically subdivide $X$ indefinitely (or just subdivide a $G$--orbit of hyperplanes);
     \item given a vertex $v\in X$ and a finite $\CAT$ cube complex $F$, we can $G$--equivariantly attach a copy of $F$ to 
each vertex in $G\cdot v$;
\item for every $n\geq 2$, we can \emph{breed} a $G$--orbit of hyperplanes of $X$ with an $n$--cube. This procedure is 
described in the case $n=2$ in \cite[Example~5.5]{BF2}, but the extension to a general $n$ is straightforward.
\end{itemize}

Such procedures for creating new cubulations from old ones are not very interesting because they always result in actions 
$G\acts X$ with some of the following properties:
\begin{itemize}
     \item $G\acts X$ \emph{inessentially}: there is some hyperplane $\mf w$ of $X$ and some component $\mf h$ of $X-\mf w$ 
such that each $G$--orbit intersects $\mf h$ in a set at bounded distance from $\mf w$;
\item $G$ acts \emph{hyperplane-inessentially}: there is some hyperplane $\mf w$ such that the action of 
$\mathrm{Stab}_G(\mf w)$ on $\mf w$ is inessential;
\item $X$ contains two hyperplanes $\mf w_1,\mf w_2$ that have associated 
halfspaces $\mf h_1,\mf h_2$ lying at finite Hausdorff distance from each other.
\end{itemize}

If $G$ is hyperbolic and acts properly and cocompactly on several essential $\CAT$ cube complexes $X_1,\dots,X_k$, 
there is an additional ``cheap'' procedure to create new cubulations. We can cubulate $G$ using all of the 
codimension--$1$ subgroups 
arising as hyperplane-stabilisers in the various $G\acts X_i$. Again, the resulting action $G\acts X$ can fail to be hyperplane-essential,
even if all original actions $G\acts X_i$ had this property.

As discussed above, we wish to restrict to essential actions, in light of \cite[Proposition 3.5]{CS}.  A procedure called \emph{panel 
collapse} reduces an arbitrary hyperplane-inessential cocompact cubulation to a 
hyperplane-essential one \cite{Hagen-Touikan}, so it makes sense to consider only hyperplane-essential actions.

This motivates the following definition:

\begin{defn}[Bald]\label{defn:bald}
A \emph{bald cubulation} of a group $G$ is a proper, cocompact, essential, hyperplane-essential action 
by cubical automorphisms on a $\CAT$ cube complex $X$ with the following additional property. 
Suppose that $\mf w_1,\mf w_2$ are hyperplanes of $X$ that bound halfspaces at finite Hausdorff distance.
Then there is a cubical splitting
$X=\R\x Y$ such that $\mf w_1$ and $\mf w_2$ are hyperplanes of the $\R$--factor.
\end{defn}

As an example, it follows from the Cubical Flat Torus theorem of Wood\-house--Wise~\cite{WW} that, for every bald cubulation of a free abelian group, 
the underlying $\CAT$ cube complex is isomorphic to the standard tiling of $\R^n$ with vertex set the integer lattice (Proposition~\ref{bald Z^n}).

To make an arbitrary cocompact cubulation bald, we \emph{shave} it. This is Proposition~\ref{facts about panel collapse}, the core
of which lies in \cite{Hagen-Touikan}.

\begin{propintro}\label{bald exist intro}
Every cocompactly cubulated group admits a bald cubulation.
\end{propintro}

Because of the uninteresting procedures described above, every group $G$ has infinitely many different cocompact
cubulations, as soon as it has one. It is a much less trivial question whether $G$ admits infinitely many different
\emph{bald} cubulations. The reason is that shaving --- in particular, panel collapse --- can radically shrink
the hyperplane-stabilisers. We are not aware of any general technique to ensure that different essential cubulations
will not get shaved into the same bald cubulation.

In fact, there can be no general procedure to produce distinct bald cubulations of a cocompactly cubulated group.
In Proposition~\ref{unique bald cubulation} we show:

\begin{propintro}\label{bald intro}
Let $T_1,T_2$ be locally finite trees with all vertices of degree at least $3$.  Let 
$U_1,U_2\leq\Aut(T_1),\Aut(T_2)$ be closed, locally primitive subgroups generated by edge-stabilisers and satisfying Tits' 
independence property.  Let $\Gamma\leq U_1\times U_2$ be a uniform lattice with dense projections to $U_1$ and $U_2$.  
Then the standard action of $\Gamma$ on $T_1\times T_2$ is the only bald cubulation of $\Gamma$.
\end{propintro}

Proposition~\ref{bald intro} implies that the irreducible lattices in products of trees constructed by 
Burger--Mozes in~\cite{BM1,BM2} have a unique bald cubulation (see in particular \cite[Theorem~6.3]{BM2}).  

An interesting question is to what extent the hypotheses of Proposition~\ref{bald intro} can be relaxed.  Specifically, a BMW group (for 
Burger--Mozes--Wise) is a group $\G$ admitting a free, vertex-transitive action $\G\to\Aut(T_1)\times\Aut(T_2)$, where 
$T_1,T_2$ are finite valence regular trees \cite{Caprace-BMW}.
Many examples of BMW groups have been studied, beginning with the aforementioned 
work of Burger--Mozes and contemporaneous work of Wise~\cite{Wise-thesis,Wise-trees}.  Which irreducible BMW groups have a 
unique bald cubulation?  


The proof of Proposition~\ref{bald intro} proceeds by studying the de Rham decomposition of 
$X$, where $\G\acts X$ is a bald cubulation provided by Proposition~\ref{bald exist intro}.  Combining results 
from~\cite{Tits,Shalom00} enables us to apply the superrigidity theorem of Chatterji--Fern\'os--Iozzi \cite{CFI}, which, in conjunction 
with a result of Caprace--de Medts~\cite{Caprace-DeMedts}, implies that each factor of the de Rham decomposition is a $\CAT$ 
cube complex with compact hyperplanes.  Baldness --- in particular, hyperplane-essentiality --- then implies that each 
factor is a tree.  A result in~\cite{BMZ} finally shows that $X$ is equivariantly isomorphic to the product of two trees we 
started with.

Proposition~\ref{bald intro} stands in sharp contrast to the situation for (most) hyperbolic groups:

\begin{thmintro}[Infinitely many bald cubulations]\label{infinitely many intro}
Let $G$ be a nonelementary Gromov-hyperbolic group acting properly and cocompactly on a $\CAT$ cube complex.  Suppose that 
{\bf \emph{at least one}} of the following holds:
\begin{itemize}
     \item $G$ is not virtually free;
\item $\Out(G)$ is infinite.
\end{itemize}
Then $G$ admits infinitely many pairwise-inequivalent bald cubulations.
\end{thmintro}

The case where $\Out(G)$ is infinite is straightforward and is dealt with in Lemma~\ref{lem:virtually_free}.  The main 
content of Theorem~\ref{infinitely many intro} is the case where $G$ splits as a (possibly trivial) finite graph of groups 
with finite edge groups and at least one vertex group one-ended.

Groups to which the theorem applies include fundamental groups of hyperbolic surfaces, fundamental groups of 
hyperbolic $3$--manifolds~\cite{Kahn-Markovic,Bergeron-Wise}, non-virtually free groups with finite $C'(\frac16)$ 
presentations~\cite{Wise-GAFA} (and hence random groups at sufficiently low density in Gromov's 
model~\cite{Ollivier-Wise}), non-virtually free hyperbolic Coxeter groups~\cite{Niblo-Reeves}, hyperbolic free-by-cyclic 
groups~\cite{Hagen-Wise1}, non-virtually free one-relator groups with torsion~\cite{Wise-qch}, Bourdon groups \cite{Bourdon-GAFA}, 
and others.  

The theorem does not apply in the case where $G$ is virtually free and $\Out(G)$ is finite; such groups were 
characterised by Pettet~\cite{Pettet} and include certain right-angled Coxeter groups (see~\cite[Proposition 
5.5]{Healy} and, more generally, \cite[Theorem~1.1]{LS} and \cite[Theorem~1.4]{GPR}).  
We discuss the virtually free case more below. 

\begin{rmk*}[Strategy of the proof of Theorem~\ref{infinitely many intro}]  Lemmas~\ref{lem:virtually_free} 
and Lemma~\ref{assembling cubulations} reduce the claim to the case where $G$ is one-ended.  

The strategy in the one-ended case is as follows.  First, we assume for a contradiction that $G$ admits only finitely many 
bald cubulations.  From this, in Lemma~\ref{minimal hyperplane}, we deduce that some such cubulation $G\acts X$ has a 
hyperplane $\mf w$ whose limit set in $\partial_\infty G$ is ``minimal'', in the sense that it does not properly contain the 
limit set of any hyperplane of any bald cubulation.  

We choose an infinite-order element $g\in G$ whose axis is cut by 
$\mf w$.  We then apply hyperplane-bending along $\mf w$ to produce crooked hyperplanes $\mf u_n$ with three key properties:
\begin{enumerate}
     \item[(i)] the fixed points of $g$ in $\partial_\infty G$ lie in different components of the complement in 
$\partial_\infty G$ of the limit set of $\mf u_n$;
\item[(ii)] the limit set of $\mf w$ is not contained in that of $\mf u_n$;
\item[(iii)] every neighbourhood in $\partial_\infty G$ of the limit set of $\mf w$ contains the limit set of $\mf u_n$ for all 
sufficiently large $n$.
\end{enumerate}

For each $n$, we add to the hyperplanes of $X$ the $G$--orbit of $\mf u_n$ and cubulate, to get a new cubulation $G\acts 
X_n$. Shaving these cubulations, we obtain bald cubulations $G\acts Y_n$ with the property that there exist 
hyperplanes $\mf{v}_n\cu Y_n$ cutting the axis of $g$, whose limit set is contained in that of $\mf{u}_n$. 
This is our only form of control on how shaving affects hyperplane-stabilisers.

By property~(iii), the limit sets of the $\mf v_n$ must Hausdorff-converge in $\partial_{\infty}G$ to a subset of the limit set
of $\mf{w}$. From the assumption that there are only finitely many bald cubulations, we see that there 
are only $\langle g\rangle$--finitely many limit sets $\mf v_n$. From this it is straightforward to conclude that one of the
$\mf{v}_n$ has its limit set contained in that of $\mf{w}$ and, from property~(ii), this is a proper inclusion. This
contradicts the ``minimality'' of $\mf{w}$. 

In this argument, it is crucial that $\mf v_n$ have nonempty limit set, which is where 
one-endedness of $G$ comes in: no cubulation of a one-ended group can have a bounded hyperplane.

\end{rmk*}

\subsection*{Further questions.}
Our results and techniques raise various questions.  

The application of hyperplane-bending used to prove Theorem~\ref{one orbit intro} in the one-ended case, and its 
combination with the ideas from~\cite{Wise-antenna} in the general case, does not in general yield a hyperplane-essential 
cubulation.  This raises the following question:

\begin{quest}\label{quest:one_orbit_HE}
Does there exist a hyperbolic group $G$ such that no hyperplane-essential cubulation of $G$ has a single orbit of 
hyperplanes? Does there exist such a $G$ that is one-ended?     
\end{quest}

When $G$ is a free group, the ``exotic'' cubulations from~\cite{Wise-antenna} are not hyperplane-essential. They are thus 
susceptible to the panel collapse procedure from~\cite{Hagen-Touikan} (summarised in Proposition~\ref{facts about panel 
collapse} below), which shrinks the hyperplane-stabilisers.  It is unknown whether every cubulation of a free group with a single 
orbit of hyperplanes panel collapses to a tree.

In some cases, the proof of Theorem~\ref{infinitely many intro} relies on twisting a fixed cubulation by the action of 
$\Out(G)$, and in other cases, it does not.  This motivates the following question:

\begin{quest}\label{quest:infinite_up_to_out}
     Let $G$ be a cocompactly cubulated one-ended hyperbolic group.  Are there infinitely many bald 
cubulations up to the 
$\Out(G)$--action?  
\end{quest}

When $G$ has no two-ended splitting, $\Out(G)$ is finite~\cite[Corollary~1.3]{BF-stable}, and one gets a positive answer from 
Theorem~\ref{infinitely many intro}.  More generally, by Levitt's characterisation of hyperbolic 
groups with infinite outer automorphism group~\cite[Theorem 1.4]{Levitt}, $\Out(G)$ is finite provided $G$ 
does not split over a two-ended subgroup with infinite centre.  So, for example, there are examples of 
hyperbolic right-angled Coxeter groups $G$ that split over $D_{\infty}$ but have $\Out(G)$ finite, and so 
Theorem~\ref{infinitely many intro} applies to give a positive answer to the question in those cases.

One can ask more refined versions of the question by measuring the complexity of bald cubulations $G\acts  X$ in some way, 
and then asking if there are infinitely many bald cubulations, up to the action of $\Out(G)$, with at most a given 
complexity.  Examples of complexity include the dimension of $X$, the number of $G$--orbits of hyperplanes in $X$, etc.

When $\Out(G)$ is infinite, one can often make fixed elements $g\in G$ have arbitrarily large translation 
length in bald cubulations of $G$.  This motivates:

\begin{quest}
Given a cocompactly cubulated hyperbolic group that is not virtually free, can each infinite-order element $g\in G$ become 
arbitrarily long in the bald cubulations of $G$?
\end{quest}

Finally, the groups we have shown to admit unique bald cubulations are not virtually special (irreducible BMW groups).  This motivates:

\begin{quest}
Which (non-hyperbolic) virtually special groups admit infinitely many bald cubulations?
\end{quest}

Among hyperbolic cocompactly cubulated groups, in view of Theorem~\ref{infinitely many intro}, the only remaining 
question is about virtually free groups with finite outer automorphism groups.  

We observe that if $G$ is a virtually free group and $\Out(G)$ is finite, then the existence of infinitely many bald 
cubulations of $G$ will require the construction of bald cubulations $G\acts X$ where $X$ has some infinite 
hyperplane-stabilisers.  Indeed, if all hyperplane-stabilisers are finite, then baldness implies that $X$ is a tree.  All of 
the proper, cocompact actions of $G$ on trees belong to the same deformation space $\mathcal D$ in the sense 
of~\cite{Guirardel-Levitt}, and $\mathcal D$ is $\Out(G)$--finite (up to projectivising), by~\cite[Proposition 
8.6]{Guirardel-Levitt}.  So it appears some other idea is needed, possibly along the lines of the proof of 
Theorem~\ref{infinitely many intro}.  

\subsection*{Outline of the paper}
In Section~\ref{sec:prelim}, we first discuss background on $\CAT$ cube complexes and cubulating groups.  We then prove 
various technical lemmas which will be used later.  We also discuss the notion of an \emph{abstract hyperplane}.  The 
procedure for ``shaving'' a $\CAT$ cube complex into a bald one is also discussed in this 
section, proving Proposition~\ref{bald exist intro}.  In Section~\ref{sec:bending}, we describe hyperplane-bending, and 
also generalise Wise's \emph{antenna} construction, to prove Theorem~\ref{one orbit intro}.  In 
Section~\ref{sec:bald_hyperbolic}, we prove Theorem~\ref{infinitely many intro}, and in Section~\ref{sec:unique_bald} we 
prove Proposition~\ref{bald intro}.

\subsection*{Acknowledgements.} 
We are grateful to Pierre-Emmanuel Caprace for directing us to Lemma~1.4.7 in \cite{BMZ}, to Anthony Genevois 
for pointing out \cite{Levitt}, and to Daniel Groves, Jason Manning and Henry Wilton for discussions. 
We also thank the anonymous referee for their many helpful comments.

Fioravanti thanks Max Planck Institute for Mathematics in Bonn for its hospitality and financial support.

\section{Preliminaries.}\label{sec:prelim}
For basic notions related to $\CAT$ cube complexes, we direct the reader to 
e.g.~\cite{Chepoi,Haglund,Haglund-GD,Sageev,Sag95,Wise-book}.  We recall some of these presently.

Throughout this section, $X$ denotes a $\CAT$ cube complex.

\subsection{$\CAT$ cube complexes.}\label{sec:cube_complex_prelim}

\subsubsection{Hyperplanes, halfspaces, separation, transversality}
We denote by $\mscr{W}(X)$ the set of hyperplanes of $X$ 
and by $\mscr{H}(X)$ the set of halfspaces.  For each $\mf w\in\mscr{W}(X)$, the two components of $X-\mf w$ are the 
halfspaces $\mf h,\mf h^*$ \emph{associated} to $\mf w$.  Each $\mf h\in\mscr{H}(X)$ is associated to (bounded by) a unique 
hyperplane $\mf w$, and $\mf h^*$ always denotes the other halfspace associated to $\mf w$.

Given $\mf w\in\mscr{W}(X)$ and $A,B\subseteq X$, we say that $\mf w$ \emph{separates} $A$ and $B$ if there is a halfspace 
$\mf h$ associated to $\mf w$ such that $A\subseteq \mf h$ and $B\subseteq\mf h^*$.  Let $\mscr{W}(A|B)$ denote the set of 
hyperplanes $\mf w$ separating $A$ from $B$. For ease of notation, we will write $\mscr{W}(x|y,z)$, rather than $\mscr{W}(x|\{y,z\})$.
   
Hyperplanes $\mf u,\mf w$ are \emph{transverse} if they are distinct and satisfy $\mf u\cap\mf w\neq\emptyset$.  
Equivalently, letting $\mf a,\mf b$ be halfspaces associated to $\mf u,\mf w$ respectively, each of the four intersections 
$\mf a\cap\mf b,\mf a^*\cap\mf b,\mf a^*\cap\mf b^*,\mf a\cap\mf b^*$ is nonempty. We also say that the halfspaces $\mf{a}$ and $\mf{b}$ are \emph{transverse}.

\begin{defn}[Facing triple, chain]\label{defn:facing_triple}
The pairwise disjoint hyperplanes $\mf u,\mf v,\mf w$ of $X$ form a \emph{facing triple} in $X$ if no two of $\mf u,\mf 
v,\mf w$ are separated by the third; equivalently, there exist disjoint halfspaces $\mf a,\mf b,\mf c$ respectively 
associated to $\mf u,\mf v,\mf w$.

The distinct hyperplanes $\mf w_1,\mf w_2,\dots$ form a \emph{chain} if, for each $i$, there exists a 
halfspace $\mf h_i$ associated to $\mf w_i$ such that (up to relabelling), we have $\mf h_i\subsetneq\mf 
h_{i+1}$ for all $i$.
\end{defn}

For each $\mf w\in\mscr{W}(X)$, recall that $\mf w$ is a $\CAT$ cube complex whose cubes are \emph{midcubes} of cubes $c$ of 
$X$ with $c\cap \mf w\neq\emptyset$.  Accordingly, we will sometimes abuse language and refer to a ``vertex of $\mf w$'' --- 
by this we mean a $0$--cube of $\mf w$ when the latter is regarded as a cube complex; equivalently, vertices of $\mf w$ are 
midpoints of edges of $X$ dual to $\mf w$.  The hyperplanes of $\mf w$ are exactly the subspaces $\mf w\cap\mf u$, as $\mf 
u$ varies over the hyperplanes of $X$ transverse to $\mf w$.

\subsubsection{The $\ell_1$ metric}
In this paper, we always work with the $\ell_1$ metric on $X$, which we denote $d$.  We will only ever be interested in 
distances between vertices of $X$, or between vertices of the cubical subdivision of $X$.  Accordingly, we just need the 
following facts about $d$:
\begin{itemize}
     \item If $x,y\in X^{(0)}$, then $d(x,y)=\#\mscr{W}(x|y)$.
     \item The metric $d$ restricts on $X^{(0)}$ to the metric induced by the usual graph metric on $X^{(1)}$.
     \item In particular, combinatorial geodesics in $X^{(1)}$ are exactly combinatorial paths containing at most one edge 
intersecting each hyperplane.
\end{itemize}

\subsubsection{The median}
Recall from e.g.~\cite{Chepoi} that a graph $\Gamma$ is \emph{median} if there exists a ternary operator 
$\mu\colon (\Gamma^{(0)})^3\to\Gamma^{(0)}$ such that (letting $d$ denote the usual graph metric), we have 
$d(x_i,x_j)=d(x_i,\mu(x_1,x_2,x_3))+d(x_j,\mu(x_1,x_2,x_3))$ for $i\neq j$ and all vertices $x_1,x_2,x_3$.  A \emph{discrete 
median algebra} is the vertex-set of a median graph, equipped with the median operator.  (This is not the standard 
definition, but it is equivalent by~\cite[Proposition 2.17]{Roller}.)

By~\cite[Theorem 6.1]{Chepoi}, $X^{(1)}$, with the graph-metric $d$, is a median graph, and conversely each median graph is 
the $1$--skeleton of a uniquely determined $\CAT$ cube complex.  Letting $\mu$ denote the median on $X^{(0)}$, we have for 
all $x,y,z\in X^{(0)}$ that $\mscr{W}(x|\mu(x,y,z))=\mscr{W}(x|y,z)$.

Fixing $p\in X^{(0)}$, the Gromov product at $p$ therefore satisfies $(x\cdot y)_p=\#\mscr{W}(p|x,y)$.  Indeed: 
\[(x\cdot y)_p=d(p,\mu(p,x,y))=\#\mscr{W}(p|\mu(p,x,y))=\#\mscr{W}(p|x,y).\]

\subsubsection{Convexity, gate-projection, and bridges}
A subset $S\subseteq X^{(0)}$ is \emph{convex} if $\mu(x,y,z)\in S$ for all $x,y\in S$, $z\in X$.  A 
subcomplex $Y$ of $X$ is \emph{convex} if $Y^{(0)}$ is convex and $Y$ is \emph{full}, in the sense that $Y$ contains every 
cube $c$ of $X$ for which $c^{(0)}\subseteq Y^{(0)}$. Equivalently, $Y$ is the largest 
subcomplex contained in the intersection of all halfspaces containing $Y$.  (For subcomplexes, this notion 
of convexity agrees with CAT(0) metric convexity~\cite{Haglund}.)   

If $Y\cu X$ is a convex subcomplex, then any combinatorial geodesic with endpoints on $Y$ lies in $Y$. Moreover, $Y$ is itself a $\CAT$ cube complex. We identify $\mscr{W}(Y)$ with the subset of $\mscr{W}(X)$ of hyperplanes that intersect $Y$.

Given $A\subseteq X$, its \emph{cubical convex hull} is the intersection of all convex subcomplexes containing $A$.  It is 
common to use the term \emph{interval} to refer to the set $I(x,y)$ of vertices $z$ such that $\mu(x,y,z)=z$, where 
$x,y\in X^{(0)}$.  The interval $I(x,y)$ is just the convex hull of $\{x,y\}$.

If $Y\subseteq X$ is a convex subcomplex, there is a \emph{gate-projection} $\pi_Y\colon X^{(0)}\to Y^{(0)}$ characterised by 
the property that any hyperplane $\mf w$ separates $x\in X^{(0)}$ from $\pi_Y(x)$ if and only if $\mf w$ separates $x$ from 
$Y$.  If $x,y\in X^{(0)}$, then $\mscr{W}(\pi_Y(x)|\pi_Y(y))=\mscr{W}(x|y)\cap\mscr{W}(Y)$, so $\pi_Y$ is $1$--Lipschitz.

The vertex $\pi_Y(x)$ is the unique closest point of $Y$ to $x$.  In fact, one can extend $\pi_Y$ to a cubical map 
$\pi_Y\colon X\to Y$; see~\cite[Section 2.1]{HHSI}.  

Let $Y,Z$ be convex subcomplexes of $X$.  Then $\pi_Y(Z)$ is a convex subcomplex of $X$, and the hyperplanes intersecting 
$\pi_Y(Z)$ are exactly the hyperplanes intersecting both $Y$ and $Z$.  In particular, if there is no such hyperplane, 
$\pi_Y(Z)$ is a single vertex.  

The convex hull $B(Y,Z)$ of $\pi_Y(Z)\cup\pi_Z(Y)$ is therefore a $\CAT$ cube complex whose 
set of hyperplanes has the form $(\mscr{W}(Y)\cap\mscr{W}(Z))\sqcup \mscr{W}(Y|Z)$.  By e.g.~\cite[Proposition 2.5]{CS}, 
we get $B(Y,Z)\cong \pi_Y(Z)\times H\cong\pi_Z(Y)\times H$, where $H$ 
is isomorphic to the interval $I(\pi_Z(y),\pi_Y(\pi_Z(y)))$ for any vertex $y\in Y$.  In particular, $\pi_Y(Z)$ and 
$\pi_Z(Y)$ are isomorphic $\CAT$ cube complexes; the maps $\pi_Y,\pi_Z$ restrict to cubical isomorphisms on these sets. The subcomplex $B(Y,Z)$ is the disjoint union of the intervals $I(y,z)$ as $(y,z)$ varies over the pairs in $Y\times Z$ with $d(y,z)=d(Y,Z)$. We refer to $B(Y,Z)$ as the \emph{bridge} between $Y$ and $Z$.

\subsubsection{Walls and median subalgebras}
This will only be used in Section~\ref{sec:unique_bald}.

A \emph{subalgebra} of $X^{(0)}$ is a subset $A$ such that $\mu(a,b,c)\in A$ whenever $a,b,c\in A$.  Given a subalgebra $A$, 
a subset $B\subseteq A$ is \emph{median-convex} in $A$ if $\mu(a,b,c)\in B$ whenever $a,b\in B$ and $c\in A$.  
Note that this coincides with our usual notion of convexity when $A=X^{(0)}$.

A \emph{wall} in $A$ is a partition $A=\mf a\sqcup \mf a^*$, where $\mf a,\mf a^*$ are nonempty and median-convex in $A$. 
When $A=X^{(0)}$, such partitions always originate from hyperplanes of $X$, by intersecting $X^{(0)}$ with the two associated halfspaces. 
For a general subalgebra, we still refer to the sets $\mf a,\mf a^*$ as \emph{halfspaces}. 
Let $\mscr{W}(A)$ and $\mscr{H}(A)$ be the set of walls and the set of halfspaces of $A$.

If $S\cu X^{(0)}$ is a convex subset, and $A$ is a subalgebra of $X^{(0)}$, then $S\cap A$ is median-convex in $A$. It follows that if 
$\mf{w}\in\mscr{W}(X)$ is a hyperplane such that $A$ intersects both associated halfspaces $\mf{h},\mf{h}^*$, 
then the partition $(\mf h\cap A)\sqcup(\mf h^*\cap A)$ is a wall in $A$. By~\cite[Lemma 6.5]{Bow1}, all walls of $A$ actually arise this way. 

\subsubsection{Cubical subdivision}
Recall that $X$ admits a \emph{cubical subdivision} $X'$ --- see Definition~2.4 in~\cite{Haglund} --- which is the $\CAT$ cube 
complex constructed as follows.

Given a cube $c\cong[-\frac12,\frac12]^n$, let $c'$ be the cube complex obtained by subdividing each factor 
$[-\frac12,\frac12]$ so that it is a graph isomorphic to $K_{1,2}$ (but with edges of length $\frac12$), and taking the product cell structure.

The subdivision $X'$ is formed from $X$ by replacing each cube $c$ by $c'$.  Then $X'$ is a $\CAT$ cube complex.   

The obvious identity maps $c\to c'$ induce a map $X\to X'$; the preimage of the vertex set of $X'$ is the set of barycentres 
of cubes of $X$.  Letting $d'$ be the $\ell_1$ metric on $X'$ (regarded as an abstract $\CAT$ cube complex 
whose cubes have side length $1$), we have $d'(x,y)=2d(x,y)$.

Each hyperplane of $X$ is a convex subcomplex of $X'$, and $X'\to X$ induces a two-to-one map $\mscr{W}(X')\to\mscr{W}(X)$ in 
an obvious way.

Letting $\mu'$ denote the median on $(X')^{(0)}$, we have that $\mu'(x,y,z)=\mu(x,y,z)$ for $x,y,z\in 
X^{(0)}\subset(X')^{(0)}$.  By working in $X'$, we can thus extend the notion of convexity to subspaces of $X$ that become 
subcomplexes upon subdivision, and this is in particular true for hyperplanes and halfspaces.  In particular, it makes sense 
to talk about the gate projection $\pi_{\mf h}\colon X\to\mf h$ where $\mf h$ is a hyperplane or halfspace, the bridge between 
two hyperplanes or two halfspaces, etc.

\subsubsection{Facts about group actions}
We denote by $\Aut(X)$ the group of cubical automorphisms of $X$.  The action of $\Aut(X)$ is an isometric action on 
$(X,d)$ and an action by median isomorphisms on $(X,\mu)$. It induces natural actions on the sets $\mscr{W}(X)$ and $\mscr{H}(X)$. 

We implicitly assume all group actions $G\acts X$ to be by cubical automorphisms.
We say that a group $G$ is \emph{cubulated} if there exists a $\CAT$ cube complex $X$ and a \emph{proper} action $\rho\colon G\to\Aut(X)$.  
If, in addition, $\rho$ can be chosen to be cocompact, then we say $G$ is \emph{cocompactly 
cubulated}.

We will discuss further properties of actions later; for the moment, we recall some facts from~\cite{Haglund}.  Let 
$g\in\Aut(X)$ and let $\mf w\in\mscr{W}(X)$.  We say that $g$ \emph{has an inversion along $\mf w$} if $g\mf w=\mf w$ and 
$g\mf h=\mf h^*$, where $\mf h$ is one of the halfspaces associated to $\mf w$.  We say that $g$ acts \emph{without 
inversions} if $g$ does not have an inversion along any hyperplane, and $g$ acts \emph{stably without inversions} if $g^n$ 
acts without inversions for all $n\in\mathbb Z$.  We have:
\begin{itemize}
     \item If $g\in\Aut(X)$ acts stably without inversions and does not fix a vertex, then $g$ is \emph{combinatorially 
hyperbolic}, i.e.\ there is a combinatorial geodesic $\gamma$ preserved by $g$, on which $g$ acts as a nontrivial translation.
\item $\Aut(X)$ acts naturally on the cubical subdivision $X'$, and every $g\in\Aut(X)$ acts stably without inversions on 
$X'$. 
\end{itemize}

If $X$ is finite-dimensional, then any $g\in\Aut(X)$ has a power acting stably without inversions.
Indeed, the hyperplanes along which the powers of $g$ have inversions are pairwise transverse.

\subsection{Pocsets}\label{sec:pocsets}
A \emph{pocset} is a triple $(\mc{P},\preceq,*)$, where the pair $(\mc{P},\preceq)$ is a poset and $*$ is an order-reversing 
involution.  Two distinct elements $a,b\in\mc{P}$ are \emph{incomparable} if $a\not\preceq b$, $b\not\preceq a$ and $a\neq b^*$. We 
say that $a$ and $b$ are \emph{transverse} if $a$ and $a^*$ are incomparable with $b$ and $b^*$. The \emph{dimension} of 
$\mc{P}$ is the maximal cardinality of a subset of pairwise-transverse elements.

An \emph{ultrafilter} is a subset $\s\cu\mc{P}$ such that:
\begin{enumerate}
\item there do not exist $a,b\in\s$ with $a\preceq b^*$;
\item for every $a\in\mc{P}$, we have $\#(\s\cap\{a,a^*\})=1$.
\end{enumerate}
Equivalently, $\s$ is a maximal subset satisfying (1). We say that $\s$ is a \emph{DCC ultrafilter} if, in addition, $\s$ does not contain any infinite descending chains. We denote by $\min\s\cu\s$ the subset of $\preceq$--minimal elements. Two ultrafilters are \emph{almost equal} if their symmetric difference is finite.

For every $\CAT$ cube complex $X$, the triple $(\mscr{H}(X),\cu,*)$ is a pocset, and the notions of transversality and dimension coincide with the usual ones. For every $v\in X^{(0)}$, the set $\{\mf{h}\in\mscr{H}(X)\mid v\in\mf{h}\}$ is a DCC ultrafilter, and, if $X$ is finite-dimensional, all DCC ultrafilters arise this way (in particular, any two of them are almost equal).

Conversely, to each pair $(\mc{P},\s)$, where $\mc{P}$ is a pocset and $\s\cu\mc{P}$ is an ultrafilter, we can associate a $\CAT$ cube complex $X=X(\mc{P},\s)$. Vertices of $X$ are exactly ultrafilters on $\mc{P}$ that are almost equal to $\s$. Two vertices are joined by an edge exactly when the symmetric difference of the corresponding ultrafilters has only two elements (the minimal possible size). See \cite{Sag95,Roller,Guralnik} for details of the construction.

\begin{lem}\label{pocset lemma}
Let $\mc{P}$ be a finite-dimensional pocset. 
\begin{enumerate}
\item For every maximal subset $\tau\cu\mc{P}$ of pairwise-transverse elements, there exists a unique DCC ultrafilter $\s\cu\mc{P}$ with $\tau\cu\min\s$. 
\item For every DCC ultrafilter $\s\cu\mc{P}$, there exists a subset $\tau\cu\min\s$ such that $\tau$ is a
maximal pairwise-transverse subset of $\mc{P}$.
\item Any two DCC ultrafilters on $\mc{P}$ are almost equal.
\end{enumerate}
Let a group $\Delta$ act on $\mc{P}$ preserving the pocset structure. 
\begin{enumerate}
\setcounter{enumi}{3}
\item If there are only finitely many $\Delta$--orbits of maximal pairwise-transverse  subsets of $\mc{P}$, then the induced 
action $\Delta\acts X(\mc{P},\s)$ is cocompact.
\end{enumerate}
\end{lem}
\begin{proof}
Parts~(1)--(3) follow from Proposition~3.1 and Corollary~3.3 in \cite{Guralnik}. In particular,  $\Delta$ leaves invariant 
the almost-equality class of any DCC ultrafilter $\s$, thus inducing an action $\Delta\acts X(\mc{P},\s)$. Finally, part~(4) 
is immediate from part~(1) and~(2).
\end{proof}

\subsection{Actions, essentiality, hyperplane-essentiality, and skewering}\label{subsec:actions}
Given a group action $G\acts X$ and a hyperplane $\mf{w}\in\mscr{W}(X)$, we denote by $G_{\mf{w}}\leq G$ the stabiliser of 
$\mf{w}$. The following is e.g.~\cite[Exercise~1.6]{Sageev}.

\begin{lem}\label{cocompact hyperplanes}
Let a group $G$ act cocompactly on $X$. For every hyperplane $\mf{w}\in\mscr{W}(X)$, the action $G_{\mf{w}}\acts\mf{w}$ is 
cocompact.
\end{lem}

\begin{proof}
Let $K$ be a compact subcomplex of $X$ such that $G\cdot K=X$.  Since $K$ is compact, there are only finitely many translates 
$g_1\mf w,\dots,g_k\mf w$ of $\mf w$ that are dual to edges of $K$.   Let 
$L=\bigcup_{j=1}^kg_j^{-1}K$.  Then $L$ is compact.   
 
Let $e$ be an edge dual to $\mf w$.  Choose $g\in G$ such that $e\subseteq gK$.  Then $g^{-1}e$ is dual to $g_i\mf w$ for some 
$i\leq k$.  Let $h=g_i^{-1}g^{-1}$.  Then $e$ and $he$ are dual to $\mf w$, and $he$ is also dual to $h\mf w$.  Hence $h\in 
G_{\mf w}$.  Now, $he$ is an edge of $g_i^{-1}K$ and is thus an edge of $L$.  So $G_{\mf w}\cdot(\mf w\cap L)$ contains 
every vertex of $\mf w$.  This shows that $G_{\mf w}$ acts on $\mf w$ with finitely many orbits of vertices.

Since the above argument can be applied to the first cubical subdivision of $X$, there are finitely many $G_{\mf w}$--orbits of vertices in $\mf w$, when the latter is given the cubical structure 
coming from the cubical subdivision.  In particular, there is a $G_{\mf w}$--finite set of vertices in $\mf w$ (in the subdivision) containing the barycentre of each maximal cube of $\mf w$ (in the 
original cubical structure).  Hence $G_{\mf w}$ acts on $\mf w$ with finitely many orbits of cubes, as required. 
\end{proof}

\begin{defn}[Skewering]\label{defn:skewering}
Let $g\in\Aut(X)$ and let $\mf w\in\mscr{W}(X)$.  We say that $g$ \emph{skewers} $\mf w$ if there is a halfspace $\mf h$ 
associated to $\mf w$ and an integer $n\neq 0$ such that $g^n\mf h\subsetneq \mf h$.  In this case, we also say $g$ 
\emph{skewers} the halfspace $\mf h$.
\end{defn}

\begin{defn}[Essential stuff]\label{defn:essential_stuff}
The CAT(0) cube complex $X$ is \emph{essential} if, for each hyperplane $\mf w$, each of the associated halfspaces contains 
points in $X$ arbitrarily far from $\mf w$.  If $G\acts X$, we say that the action is \emph{essential} if, for some (hence any) 
$x_0\in X^{(0)}$, and each hyperplane $\mf w$, each of the associated halfspaces contains points in $G\cdot x_0$ arbitrarily far 
from $\mf w$. In the latter case, we also say that $X$ is \emph{$G$--essential}.

The cube complex $X$ is \emph{hyperplane-essential} if each hyperplane $\mf w$, regarded itself as a $\CAT$ cube complex, is 
essential.  The action of $G$ is \emph{hyperplane-essential} if each hyperplane $\mf w$ has the property that $G_{\mf w}$ 
acts essentially on $\mf w$.
\end{defn}

\begin{rmk}
\begin{enumerate}
\item[]
\item Suppose $X$ is finite-dimensional. By \cite[Proposition~3.2]{CS}, the action $G\acts X$ is essential if and only if every hyperplane of $X$ is skewered by an element of $G$. Similarly, $G\acts X$ is hyperplane-essential if and only if, whenever $\mf{u},\mf{w}\in\mscr{W}(X)$ are transverse, there exist $g\in G_{\mf{w}}$ skewering $\mf{u}$ and $h\in G_{\mf{u}}$ skewering $\mf{w}$.
\item If $G$ acts cocompactly, then $X$ is essential if and only if it is $G$--essential. Similarly, by Lemma~\ref{cocompact hyperplanes}, $X$ is a hyperplane-essential cube complex if and only $G\acts X$ is a hyperplane-essential action.
\end{enumerate}
\end{rmk}

The following is \cite[Proposition~2.11]{BF1} and a proof is given in~\cite[Proposition 1]{LargeFacing}.

\begin{prop}\label{corner}
Let $X$ be cocompact, locally finite, essential, hyperplane-essential and irreducible.  For any two transverse halfspaces 
$\mf{h}_1$ and $\mf{h}_2$, there exists a halfspace ${\mf{k}\cu\mf{h}_1\cap\mf{h}_2}$.
\end{prop}

\begin{rmk}[Boundaries]\label{rem:boundaries}
Given a $\CAT$ cube complex $X$, we denote by $\partial_\infty X$ its visual boundary (with the cone topology).  Since 
$\ell_1$--convex subcomplexes, hyperplanes, and halfspaces are convex in the $\CAT$ metric, we have the following.  If $A$ is 
a convex subcomplex, hyperplane, or halfspace, the inclusion $A\hookrightarrow X$ extends to a continuous injection 
$\partial_\infty A\hookrightarrow\partial_\infty X$.

Throughout this paper, we will often be in the situation where $X$ admits a proper, cocompact action by a hyperbolic group 
$G$.  In this case, $\partial_\infty X$ is $G$--equivariantly homeomorphic to the Gromov boundary $\partial_\infty G$ of $G$.
\end{rmk}

We say that $X$ is \emph{reducible} if there exist nontrivial $\CAT$ cube complexes $A,B$ with 
$X\cong A\times B$. In this case, every hyperplane of $A,B$ determines a hyperplane of $X$, thus giving rise to a partition $\mscr{W}(X)=\mscr{W}(A)\sqcup\mscr{W}(B)$. Every hyperplane in the set $\mscr{W}(A)$ is transverse to every hyperplane in $\mscr{W}(B)$. If $X$ is not reducible, we say that $X$ is \emph{irreducible}. 

We will often require the following fact about actions on $\CAT$ cube complexes, which is Proposition~2.6 in \cite{CS}.

\begin{prop}[De Rham decomposition]\label{prop:de_rham}
     Let $X$ be finite-dimensional.  Then there is a canonical decomposition 
$X=\prod_{i=1}^mX_i$, for irreducible $\CAT$ cube complexes $X_1,\dots,X_m$, which is preserved by $\Aut(X)$ (possibly permuting the factors).  Hence the canonical embedding 
$\Aut(X_1)\times\cdots\times\Aut(X_m)\hookrightarrow\Aut(X)$ has finite-index image.
\end{prop}

Later, when working with a geometric action $G\acts X$, it will often be useful to assume that $G$ is one-ended, enabling use of the 
following lemma:

\begin{lem}[One-ended cube complexes]\label{not one-ended}
Let $X$ be one-ended and essential. Then there does not exist a partition $\mscr{W}(X)=\mc{A}\sqcup\mc{B}$ such that $\mc{A},\mc{B}$ are nonempty and no element of $\mc{A}$ is transverse to an element of $\mc{B}$.
\end{lem}
\begin{proof}
Suppose for the sake of contradiction that such a partition of $\mscr{W}(X)$ exists. By Lemma~2 in \cite{Niblo-cut}, there exists a vertex $v\in X$ such that $X-\{v\}$ is disconnected. Since $X$ is essential, each connected component of $X-\{v\}$ is unbounded. This shows that $X$ has at least two ends.
\end{proof}

\begin{defn}[Halfspace-stabiliser]\label{defn:halfspace_stabiliser}
Let $G$ be a group acting on $X$.  Let $\mf w\in\mscr{W}(X)$ and let $\mf h,\mf 
h^*$ be the associated halfspaces.  The \emph{halfspace-stabiliser} $G^0_{\mf w}$ is the kernel of the natural action of 
$G_{\mf w}$ on $\{\mf h,\mf h^*\}$ (which has index at most $2$ in $G_{\mf{w}}$).
\end{defn}

Recall that a subgroup $H$ of a group $G$ is \emph{separable} if for all $g\in G- H$, there is a finite-index 
subgroup $G'\leq G$ such that $H\leq G'$ and $g\not\in G'$.

\begin{lem}[Large-girth covers]\label{girthy cover}
Let a group $G$ act properly, cocompactly and  with separable halfspace-stabilisers on $X$. Then, for every $n\geq 1$, there 
exists a finite-index subgroup $H\lhd G$ such that:
\begin{itemize}
\item $H\acts X$ has no hyperplane inversions;
\item for every $\mf{w}\in\mscr{W}(X)$, any two distinct elements of $H\cdot\mf{w}$ are disjoint and at distance $\geq n$.
\end{itemize}
\end{lem}
\begin{proof}
Let $\mf{w}_1,\dots,\mf{w}_k\in\mscr{W}(X)$ be such that $G\cdot\{\mf{w}_1,\dots,\mf{w}_k\}=\mscr{W}(X)$.  

Since 
halfspace-stabilisers are separable, 
there exist subgroups $H_i\leq G$ such 
that any two elements of $H_i\cdot\mf{w}_i$ are at distance $\geq n$ and no element of $H_i$ swaps the sides of $\mf{w}_i$.

Up to passing to further finite-index subgroups, we can assume that $H_i\lhd G$.
Set $H=H_1\cap\dots\cap H_k$.

Given any $\mf{w}\in\mscr{W}(X)$, there exist $1\leq i\leq k$ and $g\in G$ with $\mf{w}=g\mf{w}_i$.  If $h\in H$ and 
$\mf{w}\neq h\mf{w}$, we have $\mf{w}_i\neq g^{-1}hg\mf{w}_i$. Since $H$ is normal in $G$, the element $g^{-1}hg$ lies in 
$H$, hence $d(\mf{w},h\mf{w})=d(\mf{w}_i,g^{-1}hg\mf{w}_i)\geq n$. A similar argument shows that $H\acts X$ has no hyperplane 
inversions.
\end{proof}

\begin{lem}\label{disjoining hyperplanes}
Let $G$ be a Gromov-hyperbolic group, with a proper cocompact action $G\acts X$.  Given essential hyperplanes 
$\mf{w}_1,\mf{w}_2$ and $n\geq 1$, there exists $\mf{w}_2'\in G\cdot\mf{w}_2$ such that $d(\mf{w}_1,\mf{w}_2')\geq n$.
\end{lem}

\begin{proof}
By Proposition~3.2 in \cite{CS}, the orbits $G\cdot\mf{w}_1$ and $G\cdot\mf{w}_2$ contain infinite chains of 
hyperplanes; moreover, we can 
assume that $G\cdot\mf{w}_1\neq G\cdot\mf{w}_2$. If every element of $G\cdot\mf{w}_1$ were transverse to an element of 
$G\cdot\mf{w}_2$, this would violate Theorem 3.3 of~\cite{Genevois2}. It follows that some $\mf{w}_2''\in 
G\cdot\mf{w}_2$ is 
disjoint from $\mf{w}_1$ and we can achieve the required distance from $\mf{w}_1$ by considering a hyperplane 
$\mf{w}_2'=g^N\mf{w}_2''$, where $g$ skewers $\mf{w}_2'$ and $N$ is large.
\end{proof}

Given a geodesic $\g\cu X$, we denote by $\mscr{W}(\g)$ the set of hyperplanes crossed by $\g$. 

\begin{lem}\label{small projection}
There exists a constant $D=D(\delta)$ such that the following holds.
For every $\delta$--hyperbolic $\CAT$ cube complex $X$ and every geodesic $\g\cu X$, 
there exists a hyperplane $\mf{w}\in\mscr{W}(\g)$ with $\diam\pi_{\mf{w}}(\g)\leq D$. 
\end{lem}
\begin{proof}
There exists a constant $C=C(\delta)$ such that, given any two transverse chains of halfspaces, one of them must have 
cardinality $<C$ (see e.g. \cite[Theorem 3.3]{Genevois2}).  

Set $d=\dim X$, $\Delta=Cd(2d+1)$, $D=2\Delta$ and observe 
that $d$ (hence $D$) is bounded above in terms of $\delta$. Since gate-projections are $1$--Lipschitz, we can assume that the 
length of $\g$ is at least $D$.

Among any $2Cd+1$ halfspaces entered consecutively by $\g$, there exists a chain $\mf{h}_{-C}\supsetneq 
\dots\supsetneq\mf{h}_C$.  Let $x_-\in\mf{h}_0^*\cap\g$ and $x_+\in\mf{h}_0\cap\g$ be adjacent vertices of $X$ and let $\mf{w}$ 
be the hyperplane bounding $\mf{h}_0$. Note that $\mscr{W}(x_-|\mf{h}_C)$ and $\mscr{W}(x_+|\mf{h}_{-C}^*)$ each contain at 
most $2Cd$ hyperplanes (not the optimal bound). 

Suppose for the sake of contradiction that there exists $y\in\g\cap\mf{h}_C$ with 
$d(\pi_{\mf{w}}(x_-),\pi_{\mf{w}}(y))>\Delta$.  Then there exists a chain $\mf{k}_1\supsetneq \dots\supsetneq\mf{k}_{C(2d+1)}$ of 
halfspaces with $\pi_{\mf{w}}(x_-)\in\mf{k}_1^*$ and $\pi_{\mf{w}}(y)\in\mf{k}_{C(2d+1)}$. For all $i,j\geq 1$, we have 
$y\in\mf{h}_i\cap\mf{k}_j$, $\pi_{\mf{w}}(y)\in\mf{h}_i^*\cap\mf{k}_j$ and $x_-\in\mf{h}_i^*\cap\mf{k}_j^*$. Thus, either 
$\mf{h}_i\subsetneq\mf{k}_j$ or $\mf{h}_i$ and $\mf{k}_j$ are transverse. If $j>2Cd$, the halfspaces $\mf{h}_C$ and 
$\mf{k}_j$ must be transverse, as $\#\mscr{W}(x_-|\mf{h}_C)\leq 2Cd$. It follows that $\mf{h}_i$ and $\mf{k}_j$ are 
transverse for all $1\leq i\leq C$ and $2Cd+1\leq j\leq 2Cd+C$, violating our choice of $C$.

This proves that $d(\pi_{\mf{w}}(x_-),\pi_{\mf{w}}(y))\leq\Delta$ for every $y\in\g\cap\mf{h}_C$ and a  similar argument 
shows that $d(\pi_{\mf{w}}(x_+),\pi_{\mf{w}}(z))\leq \Delta$ for all $z\in\g\cap\mf{h}_{-C}^*$. We conclude that the 
projection $\pi_{\mf{w}}(\g\cap(\mf{h}_C\sqcup\mf{h}_{-C}^*))$ is contained in the $\Delta$--neighbourhood of 
$\pi_{\mf{w}}(x_-)=\pi_{\mf{w}}(x_+)$. Since $\pi_{\mf{w}}(\g)$ is an (unparametrised) geodesic, it all lies in the 
$\Delta$--neighbourhood of $\pi_{\mf{w}}(x_+)$, hence $\diam\pi_{\mf{w}}(\g)\leq 2\Delta=D$.
\end{proof}

\subsection{Hyperbolic groups and abstract hyperplanes.}\label{section:hypgps}

Let $G$ be a Gromov-hyperbolic group. Every infinite-order element $g\in G$ has exactly two fixed points in the Gromov 
boundary $\partial_{\infty}G$.  We denote by $g^+$ the stable fixed point and by $g^-$ the unstable one. The following is 
classical:

\begin{lem}\label{density}
The set $\{(g^+,g^-)\mid g \text{ infinite-order}\}$ is dense in $\partial_{\infty}G\x\partial_{\infty}G$.
\end{lem}

If $H\leq G$ is a quasiconvex subgroup, we denote by $\L H \cu\partial_{\infty}G$ its limit set.  The following is Lemma~2.6 
in \cite{GMRS}:

\begin{lem}\label{intersection of limit sets}
We have $\L(H\cap K)=\L H\cap\L K$ for any two quasiconvex subgroups $H,K\leq G$.
\end{lem}

If $G$ acts properly and cocompactly on a geodesic metric space $X$, there exists a unique $G$--equivariant homeomorphism 
$\phi\colon\partial_{\infty}G\ra\partial_{\infty}X$. Given a 
subset $\Om\cu X$, we denote by 
$\partial_{\infty}\Om\cu\partial_{\infty}X$ its limit set in the visual boundary of $X$ and by 
$\L\Om=\phi^{-1}(\partial_{\infty}\Om)\cu\partial_{\infty}G$ the pull-back to $G$.

Let us now fix a Cayley graph $\G(G)$ of $G$ with respect to a finite generating set. Given a subgroup $H\leq G$, we denote by $N_r(H)$ its closed $r$--neighbourhood in $\G(G)$.

The following is Lemma~7.3 in \cite{Hruska-Wise-finiteness},  although it originally appeared implicitly in 
\cite{Sag97,GMRS}. See also Lemma~7 in \cite{Niblo-Reeves} and Theorem~1.1 in \cite{Hruska-Wise-packing} for additional 
details on its proof.

\begin{lem}\label{intersection of cosets}
Given $D,\kappa\geq 0$ there exists a constant $C$ such that the following holds for all $\kappa$--quasiconvex 
subgroups $H_1,\dots,H_k\leq G$. If 
\[N_D(g_1H_{i_1}),\dots,N_D(g_nH_{i_n})\] 
pairwise intersect, there exists $g\in G$ that is $C$--close to all $g_1H_{i_1},\dots,g_nH_{i_n}$.
\end{lem}

We will need the following result in Section~\ref{infinitely-ended sect}.

\begin{lem}\label{connected components in general}
Let $X$ be a locally connected, 
proper, geodesic, $\delta$--hyperbolic space and, for some $R\geq 0$, let $U\cu X$ be a closed $R$--quasiconvex subset. 
For every subset $A\cu X- U$, let us set $\tilde A:=A\sqcup U$. Then:
\begin{enumerate}
\item if $A$ is a union of connected components of $X- U$, the set $\tilde A$ is $R$--quasiconvex;
\item if $A$ and $B$ are unions of connected components of $X- U$ such that $A\cap B=\emptyset$, then $\partial_{\infty}\tilde A\cap\partial_{\infty}\tilde B=\partial_{\infty}U$;
\item $\partial_{\infty}X-\partial_{\infty}U$ is a disjoint union of the open subsets $\partial_{\infty}\tilde C-\partial_{\infty}U$, where $C$ is a connected component of $X- U$.
\end{enumerate}
\end{lem}
\begin{proof}
In order to prove part~(1), consider a geodesic $\g\cu X$ joining two points $x,y\in\tilde A$. Let $x',y'\in\g$ be the points furthest from $x$ and $y$, respectively, such that the sub-segments $xx',yy'\cu\g$ are entirely contained in the closure $\overline A\cu X$. If $x$ or $y$ do not lie in $\overline A$, we set $x'=x$ or $y'=y$. The points $x'$ and $y'$ always lie in $U$, so the sub-segment of $\g$ joining them is contained in the $R$--neighbourhood of $U\cu\tilde A$. Since the rest of $\g$ is contained in $\overline A$, the entire $\g$ lies inside the $R$--neighbourhood of $\tilde A$, showing part~(1).

Given $A,B$ as in part~(2) and a point $\xi\in\partial_{\infty}\tilde A\cap\partial_{\infty}\tilde B$, we consider quasi-geodesic rays $r_A\cu\tilde A$ and $r_B\cu\tilde B$ representing $\xi$. Let $x_n\in r_A$ and $y_n\in r_B$ be diverging sequences with $\sup d(x_n,y_n)<+\infty$. Observing that any geodesic joining $x_n$ to $y_n$ must intersect $U$, we deduce that $r_A$ and $r_B$ stay at bounded distance from $U$. Hence $\xi\in\partial_{\infty}U$, which proves part~(2).

Now, let $\mscr{C}$ be the set of all connected components $C\cu X- U$. Given a point $\xi\in\partial_{\infty}X-\partial_{\infty}U$ and a ray $r\cu X$ representing it, a sub-ray $r'\cu r$ must be disjoint from $U$. It follows that $r'$ is contained in some $C\in\mscr{C}$, hence $\xi\in\partial_{\infty}\tilde C-\partial_{\infty}U$. This shows that $\partial_{\infty}X$ is the union of the sets $\partial_{\infty}\tilde C-\partial_{\infty}U$ with $C\in\mscr{C}$; by part~(2), this is a disjoint union. Finally, observe that, for each $C\in\mscr{C}$, the boundary $\partial_{\infty}X$ is union of the two closed subsets $\partial_{\infty}\tilde C$ and $\partial_{\infty}(X- C)$. Again by part~(2), this shows that $\partial_{\infty}\tilde C-\partial_{\infty}U$ has closed complement in $\partial_{\infty}X-\partial_{\infty}U$ and is therefore open.
\end{proof}

The following is the key notion in this subsection. It allows us to cubulate hyperbolic groups with as few non-canonical 
choices as possible.

\begin{defn}\label{abstract hyperplanes defn}
An \emph{abstract hyperplane} for a hyperbolic group $G$ is a pair $(H,\mf{H})$, where $H\leq G$ is quasiconvex and $\mf{H}$ is an $H$--invariant 
partition $\partial_{\infty}G-\L H=\mf{H}^+\sqcup\mf{H}^-$, where $\mf{H}^{\pm}$ are nonempty open subsets. 
\end{defn}

We say that $\mf{H}^{\pm}$ are the \emph{sides} of $\mf{H}$ and that $\mf{H}$ is \emph{subordinate} to $H$.  Two points 
$\xi,\eta\in\partial_{\infty}G$ are \emph{separated} by $\mf{H}$ if they lie on opposite sides. 

Observe that, by $H$--invariance, the 
closures $\overline{\mf{H}^+},\overline{\mf{H}^-}\cu\partial_{\infty}G$ are exactly $\mf{H}^+\sqcup\L H$ and 
$\mf{H}^-\sqcup\L H$.

\begin{lem}\label{abstract hyperplanes vs almost invariant sets}
Given any abstract hyperplane $(H,\mf{H})$, there exist two $H$--invariant open subsets $H^{\pm}\cu\G(G)$ and a constant $D>0$ such that:
\begin{enumerate}
\item $\G(G)- N_D(H)=H^+\sqcup H^-$;
\item $H^+\sqcup N_D(H)$, $H^-\sqcup N_D(H)$ and $N_D(H)$ are connected;
\item $\L H^+=\overline{\mf{H}^+}$ and $\L H^-=\overline{\mf{H}^-}$.
\end{enumerate}
\end{lem}
\begin{proof}
Given $L>0$, we denote by $A_L^+\cu\G(G)$ the closed $L$--neighbourhood of the weak hull of $\overline{\mf{H}^+}=\mf{H}^+\sqcup\L 
H$.  (Recall that the weak hull of $\overline{\mf{H}^+}$ is the union of all $\Gamma(G)$--geodesics joining 
distinct points in $\overline{\mf{H}^+}$.)

The set $A_L^-$ is defined similarly. Observe that, for every sufficiently large value of $L$, there exists 
$D>0$ such that:
\begin{itemize}
\item $\L A_L^{\pm}=\overline{\mf{H}^{\pm}}$ 
and $\G(G)=A_L^+\cup A_L^-$; 
\item $A_L^+\cap A_L^-\cu N_D(H)$; 
\item the sets $A_L^+\cup N_D(H)$, $A_L^-\cup N_D(H)$ and $N_D(H)$ are connected. 
\end{itemize}
Thus, the sets $H^+=A_L^+- N_D(H)$ and $H^-=A_L^-- N_D(H)$ are open, 
$H$--invariant, and satisfy~(1) and~(2).  It is clear from the construction that $\mf{H}^+\cu\L H^+\cu\L A_L^+=\mf{H}^+\sqcup\L H$. 
Since $H^+$ is nonempty and $H$--invariant, we also have $\L H\cu\L H^+$, which shows (3).
\end{proof}

\begin{rmk}
Lemma~\ref{abstract hyperplanes vs almost invariant sets} shows that, if there exists an abstract hyperplane subordinate to 
$H$, then $H$ must be a codimension-one subgroup of $G$. Conversely, it is not hard to see that, for every quasiconvex 
codimension-one subgroup $H\leq G$, there exist abstract hyperplanes subordinate to $H$.
\end{rmk}

\begin{rmk}\label{comparing transversality}
Let $(H,\mf{H})$, $(K,\mf{K})$ be abstract hyperplanes and let $H^{\pm}, K^{\pm}$ be the sets constructed in Lemma~\ref{abstract hyperplanes vs almost invariant sets}. The constant $D$ can always be enlarged, so, without loss of generality, it is the same for both. If $N_D(H)$ and $N_D(K)$ are disjoint, then a side of $\mf{H}$ is disjoint from a side of $\mf{K}$.

Indeed, since $N_D(H)$ is connected, we have either $N_D(H)\cu K^+$ or $N_D(H)\cu K^-$; without loss of generality, let us assume that the former holds. Similarly, we have $N_D(K)\cu H^-$ without loss of generality. It follows that the connected set $H^+\sqcup N_D(H)$ is disjoint from $N_D(K)$ and thus contained in a single connected component of its complement. Since $N_D(H)\cu K^+$, we have $H^+\cu K^+$ and $\L H^+\cu \L K^+$. Hence $\mf{H}^+\cap\mf{K}^-=\emptyset$.
\end{rmk}

The following is little more than a rephrasing in terms of $\partial_{\infty}G$ of well-known results from \cite{Sag97,Bergeron-Wise}.

\begin{prop}\label{cubulating with abstract hyperplanes}
Let $\mc{H}$ be a $G$--invariant set of abstract hyperplanes.
\begin{enumerate}
\item If $\mc{H}$ contains only finitely many $G$--orbits, $\mc{H}$ gives rise to a cocompact $G$--action on an essential $\CAT$ cube complex $X(\mc{H})$.
\item In this case, the action $G\acts X(\mc{H})$ is proper if and only if $g^+,g^-\in\partial_{\infty}G$ are separated by an element of $\mc{H}$, for every infinite-order $g\in G$.
\end{enumerate}
\end{prop}
\begin{proof}
The collection $\mc{P}=\{\overline{\mf{H}^{\eps}} \mid \mf{H}\in\mc{H},\ \eps\in\{\pm\}\}$ has a natural structure of poset coming from inclusions. We promote this to a structure of pocset by setting $(\overline{\mf{H}^+})^*=\overline{\mf{H}^-}$. 

Observe that $\overline{\mf{H}^+}\cu\overline{\mf{K}^+}$ if and only if $\mf{H}^+\cap\mf{K}^-=\emptyset$. Thus, $\overline{\mf{H}^+}$ and $\overline{\mf{K}^+}$ are transverse if and only if both $\mf{H}^+$ and $\mf{H}^-$ intersect both $\mf{K}^+$ and $\mf{K}^-$.

Since there are finitely many $G$--orbits in $\mc{H}$, Lemma~\ref{abstract hyperplanes vs almost invariant sets} provides a 
constant $D$ that works for every element of $\mc{H}$. By Remark~\ref{comparing transversality}, every set of $k$ 
pairwise-transverse elements of $\mc{P}$ corresponds to a collection of $k$ cosets of uniformly quasiconvex subgroups of $G$ whose 
$D$--neighbourhoods pairwise intersect. Lemma~\ref{intersection of cosets} shows that $\mc{P}$ is finite-dimensional and 
contains only finitely many $G$--orbits of maximal pairwise-transverse subsets. Lemma~\ref{pocset lemma} thus yields a 
natural cocompact action on a $\CAT$ cube complex $X(\mc{H})$. 

For every $\mf{H}\in\mc{H}$, Lemma~\ref{density} shows that there exists an infinite-order element $g\in G$ with $g^+\in\mf{H}^+$ and $g^-\in\mf{H}^-$. A power of $g$ must then skewer the hyperplane of $X(\mc{H})$ determined by $\mf{H}$. We conclude that $X(\mc{H})$ is essential. Finally, part~(2) follows from Proposition~1.3 in \cite{Bergeron-Wise}.
\end{proof}

\begin{rmk}\label{automatic hyperplane-essentiality}
Let $H_1,\dots,H_k$ be quasiconvex subgroups of $G$ with the property that, for each $i\leq k$, the difference 
$\partial_{\infty}G-\L H_i$ has exactly two connected components, and these are left invariant by the 
$H_i$--action. Each $H_i$ determines a unique abstract hyperplane $\mf{H}_i$ and we can consider the collection 
$\mc{H}=G\cdot\mf{H}_1\cup\dots \cup G\cdot\mf{H}_k$. In this case, the cube complex $X(\mc{H})$ provided by 
part~(1) of Proposition~\ref{cubulating with abstract hyperplanes} is automatically hyperplane-essential.

In order to see this, consider abstract hyperplanes $\mf{H},\mf{K}\in\mc{H}$, with stabilisers $H,K$ respectively. If $\L H\cap\mf{K}^+=\emptyset$, the connected set $\mf{K}^+$ is partitioned into the two open sets $\mf{K}^+\cap\mf{H}^{\pm}$. It follows that one of these two sets is empty and, in particular, $\mf{H}$ and $\mf{K}$ are not transverse.

Thus, if $\mf{H}$ and $\mf{K}$ \emph{are} transverse, the four intersections $\L H\cap\mf{K}^{\pm}$ and $\L K\cap\mf{H}^{\pm}$ must all be nonempty and open in the respective limit sets. By Lemma~\ref{density}, there exists an infinite-order element $h\in H$ with $h^+\in\mf{K}^+$ and $h^-\in\mf{K}^-$; in particular, a sufficiently large power of $h$ skewers the hyperplane of $X(\mc{H})$ determined by $\mf{K}$. Similarly, there exists $k\in K$ skewering the hyperplane determined by $\mf{H}$. This shows that the action $G\acts X(\mc{H})$ is hyperplane-essential.
\end{rmk}

\subsection{Shaving cocompact cubulations.}

Let $X$ be a $\CAT$ cube complex.

\begin{defn}\label{defn:SP}
Two distinct hyperplanes of $X$ are said to be \emph{effectively parallel} if they are disjoint and bound 
halfspaces at finite Hausdorff 
distance.
\end{defn}

As an example, the cubical subdivision $X'$ contains a pair of effectively parallel hyperplanes for every 
hyperplane of $X$. 

Note that, in general, two disjoint hyperplanes can be at finite Hausdorff distance without being effectively 
parallel.  For 
instance, in a tree with all vertices of degree $\geq 3$, any two distinct hyperplanes are at finite Hausdorff distance, but 
no two of them are effectively parallel.

We also observe that, if $X$ is hyperplane essential and $\mf{h}_1,\mf{h}_2$ are distinct halfspaces at finite Hausdorff distance,
then the hyperplanes $\mf{w}_1,\mf{w}_2$ bounding them are necessarily disjoint. Thus, $\mf{w}_1$ and $\mf{w}_2$ are 
effectively parallel.

Given a subset $A\cu\mscr{H}(X)$, we employ the notation $A^*=\{\mf{h}^*\mid\mf{h}\in A\}$.

\begin{lem}\label{structure of SP hyp}
Let $X$ be locally finite, cocompact, essential and hyperplane-essential.  Let $\mscr{P}\cu\mscr{H}(X)$ be a maximal set of 
halfspaces pairwise at finite Hausdorff distance. Then:
\begin{enumerate}
\item the subset $\mscr{P}\cu\mscr{H}(X)$ is totally ordered by inclusion;
\item $\mscr{P}^*$ is also a maximal set of halfspaces at finite Hausdorff distance;
\item if $\mf{w}\in\mscr{W}(X)$ doesn't bound an element of $\mscr{P}$, then $\mf{w}$ is either transverse to all, contained in all, or disjoint from all elements of $\mscr{P}$.
\end{enumerate}
\end{lem}
\begin{proof}
Since any two elements of $\mscr{P}$ are at finite Hausdorff distance and $X$ is hyperplane-essential, no two elements of 
$\mscr{P}$ are transverse.  Since $X$ is essential, no two elements of $\mscr{P}$ can be disjoint or have disjoint 
complements. This shows part~(1), while part~(2) is clear. We now prove part~(3).

Let $\mf{u},\mf{v}$ be hyperplanes bounding elements of $\mscr{P}$.  
By essentiality of $X$, no $\mf{w}\in\mscr{W}(X)$ can 
form a facing triple with $\mf{u}$ and $\mf{v}$. Since $\mscr{P}$ is maximal, if $\mf{w}$ does not bound an element of $\mscr{P}$, 
then $\mf{w}$ cannot separate $\mf{u}$ and $\mf{v}$ 
either. Thus, if $\mf{w}$ is not transverse to any element of $\mscr{P}$, then $\mf{w}$ must be either contained in all 
elements of $\mscr{P}$, or contained in all of their complements. 

Finally, suppose that $\mf{w}$ is transverse to an element 
of $\mscr{P}$, but not to all of them. In this case, $\mf{w}$ and the elements of $\mscr{P}$ all originate from a single de 
Rham factor of $X$ and Proposition~\ref{corner} provides a hyperplane forming a facing triple with two hyperplanes bounding 
elements of $\mscr{P}$. This is a contradiction and it concludes the proof of part~(3).
\end{proof}

Given $X$ as in Lemma~\ref{structure of SP hyp} and $\mf{h}\in\mscr{H}(X)$, we denote by $\mscr{P}(\mf{h})\cu\mscr{H}(X)$ the 
subset  of all halfspaces at finite Hausdorff distance from $\mf{h}$. We define:
\[\parclass(X)=\{\mscr{P}(\mf{h})\mid\mf{h}\in\mscr{H}(X)\}.\] 
Given distinct elements $\mscr{Q}_1,\mscr{Q}_2\in\parclass(X)$, Lemma~\ref{structure of SP hyp} shows that whether 
$\mf{h}_1\in\mscr{Q}_1$ is contained in $\mf{h}_2\in\mscr{Q}_2$ is independent of the choice of $\mf{h}_1$ and $\mf{h}_2$. 
When this happens, we write $\mscr{Q}_1\preceq\mscr{Q}_2$. We obtain a pocset $(\parclass(X),\preceq,\ast)$ and a 
surjective pocset homomorphism $\mscr{P}\colon\mscr{H}(X)\ra\parclass(X)$.

The dimension of the pocset $\parclass(X)$ coincides with $\dim X<+\infty$. Lem\-ma~\ref{pocset lemma} thus guarantees the 
existence of a unique class of DCC ultrafilters on $\parclass(X)$, which gives rise to a $\CAT$ cube complex 
$\cmp(X)$. We refer to $\cmp(X)$ as the \emph{compression} of $X$. Observe that $\parclass(X)$ and $\cmp(X)$ are naturally equipped with an $\Aut(X)$--action. 

The preimage under $\mscr{P}$ of any ultrafilter on $\parclass(X)$ is an ultrafilter on $\mscr{H}(X)$. 
Assuming for a moment that all fibres of $\mscr{P}$ are finite, preimages of DCC ultrafilters are again DCC. 
In this case, we obtain an $\Aut(X)$--equivariant injection 
$\iota\colon\cmp(X)^{(0)}\hookrightarrow X^{(0)}$, which does not shrink distances.

As the next result shows, the condition on the fibres of $\mscr{P}$ corresponds to $X$ having no $\R$--factors in its de Rham decomposition.

\begin{lem}\label{removing SP hyp}
Let $X$ be locally finite, cocompact, essential and hyperplane-essential. Assume that $X$ has no factors isomorphic to $\R$ in its de Rham decomposition. Then:
\begin{enumerate}
\item the fibres of the map $\mscr{P}$ are uniformly finite and $\iota$ is bi-Lipschitz;
\item $g\in\Aut(X)$ skewers $\mf{h}\in\mscr{H}(X)$ if and only if it skewers the halfspace of $\cmp(X)$ determined by 
$\mscr{P}(\mf{h})$;
\item $\cmp(X)$ is locally finite, cocompact, essential, hyperplane-essential and it has no halfspaces at finite Hausdorff distance.
\end{enumerate}
\end{lem}
\begin{proof}
Let $n$ be the number of orbits of the action $\Aut(X)\acts\mscr{H}(X)$. If an element $\mscr{Q}\in\parclass(X)$ contains $>n$ 
halfspaces, there exist $g\in\Aut(X)$ and $\mf{h}\in\mscr{Q}$ such that $\mf{h}\subsetneq g\mf{h}\in\mscr{Q}$. 
In this case, the halfspaces $g^n\mf{h}$ are pairwise at finite Hausdorff distance, hence 
$\{g^n\mf{h}\}_{n\in\Z}\cu\mscr{Q}$. 
Part~(1) of Lemma~\ref{structure of SP hyp} implies that $\mscr{Q}$ is a bi-infinite 
chain and, by part~(3) of Lemma~\ref{structure of SP hyp}, every halfspace in $\mscr{Q}\sqcup\mscr{Q}^*$ is transverse to all the halfspaces in 
$\mscr{H}(X)-(\mscr{Q}\sqcup\mscr{Q}^*)$. This contradicts the assumption that $X$ has no de Rham factors isomorphic 
to $\R$. We conclude that all fibres of $\mscr{Q}$ have cardinality $\leq n$, hence $\iota$ is $n$--Lipschitz. This yields
part~(1). Parts~(2) and~(3) follow immediately.
\end{proof}

We now make Definition~\ref{defn:bald} from the introduction a bit more precise (recalling that, in a hyperplane-essential cube complex, 
halfspaces at finite Hausdorff distance are always bounded by effectively parallel hyperplanes).

\begin{defn}
A $\CAT$ cube complex $X$ is \emph{bald} if it is essential, hyper\-plane-essential and, moreover, the 
following holds. If $\mf{w}_1,\mf{w}_2\in\mscr{W}(X)$ are effectively parallel, there exists a factor $L$ in 
the de Rham decomposition of $X$ such that $L\cong\R$ and $\mf{w}_1,\mf{w}_2\in\mscr{W}(L)$. 

A \emph{bald cubulation} is a proper, cocompact action on a bald cube complex.
\end{defn}

Recall that, as defined in Section~\ref{section:hypgps}, if a hyperbolic group $G$ acts properly and cocompactly on a $\CAT$ cube complex $X$, then every subset $\Om\cu X$ determines limits sets $\L\Om\cu\partial_{\infty}G$ and $\partial_{\infty}\Om\cu\partial_{\infty}X$.

\begin{prop}\label{facts about panel collapse}
Let a group $G$ act properly and cocompactly on a $\CAT$ cube complex $X$. Then there exists another $\CAT$ cube complex $X_{\bullet}$ and a proper cocompact action $G\acts X_{\bullet}$ such that:
\begin{enumerate}
\item $X_{\bullet}$ is bald;
\item if $G$ is hyperbolic and $g\in G$ skewers $\mf{w}\in\mscr{W}(X)$, there exists a hyperplane $\mf{u}\in\mscr{W}(X_{\bullet})$ such that $\mf{u}$ is skewered by $g$ and $\L\mf{u}\cu\L\mf{w}$.
\end{enumerate}
\end{prop}
\begin{proof}
 By replacing $X$ with the cubical subdivision, we can assume that $G$ acts on $X$ without 
inversions.  Let $\#(X)=(n_0,\dots,n_{\mathrm{dim} X-2})$, where for each $i$, $n_i$ is the number of 
$G$--orbits of $i$--cubes.

If $X$ is hyperplane-essential, then, by passing to the $G$--essential core, we can 
assume that the cube complex is essential and hyperplane-essential (cf.\ Proposition~3.5 in \cite{CS}).

If not, then, by Theorem A of~\cite{Hagen-Touikan}, $X$ contains a $G$--invariant subspace 
$Y$ that has the structure of a $\CAT$ cube complex, with $Y^{(0)}=X^{(0)}$.  Moreover, the set of hyperplanes 
of $Y$ is exactly the set of components of subspaces of the form $\mf u\cap Y$, where $\mf u$ is a hyperplane 
of $X$.  Finally, the action of $G$ on $Y$ is without inversions, and $\#(Y)<\#(X)$ (in lexicographic order).

Iterating finitely many times, we find a hyperplane-essential cocompact action $G\acts Z$, where $Z$ is a 
$\CAT$ cube complex $G$--equivariantly embedded in $X$; by replacing $Z$ with its $G$--essential core, we have 
that $Z$ is essential and hyperplane-essential, and $Z^{(0)}\subseteq X^{(0)}$.

Since $G$ acts on $X$ properly and $Z\hookrightarrow X$ is $G$--equivariant, each $0$--cube of $X$, and hence each $0$--cube of $Z$, has 
finite stabiliser in $G$. Since $G$ acts on $Z$ cocompactly, the action of $G$ on $Z$ is therefore 
proper. We are left to deal with effectively parallel hyperplanes, in order to ensure that $Z$ is bald. 

Isolating the $\R$--factors in the de Rham decomposition of $Z$,  we obtain a splitting $Z=\R^m\x W$, where $m\geq 0$ and $W$ 
is a $\CAT$ cube complex with no $\R$--factors. Observe that $G$ leaves invariant this decomposition of $Z$, and the induced 
action $G\acts W$ is cocompact.

We set $X_{\bullet}=\R^m\x\cmp(W)$. Observe that the $G$--action descends to $X_{\bullet}$. By part~(1) of Lemma~\ref{removing SP hyp}, $X_{\bullet}$ is $G$--equivariantly quasi-isometric to $Z$; hence $G\acts X_{\bullet}$ is proper and cocompact. By part~(3) of Lemma~\ref{removing SP hyp}, $X_{\bullet}$ is bald. This completes the proof of part~(1).

We now assume that $G$ is hyperbolic and prove part~(2). Hyperbolicity implies that $m=0$, so $Z=W$ and $X_{\bullet}=\cmp(Z)$.

Suppose that $g\in G$ skewers a hyperplane $\mf w$ of $X$.  By Theorem A 
of~\cite{Hagen-Touikan}, $\mf w\cap 
Z=\bigsqcup_{i\in\mathcal I}\mf w_i$, where each $\mf w_i$ is a hyperplane of $Z$. In particular, $\L\mf 
w_i\subseteq \L\mf w$ for each $i\in\mathcal I$.  

Let $\gamma$ be an axis for $g$ in $Z$, so that the endpoints of $\gamma$ are $g^\pm\in\partial_\infty Z$.  
Suppose that no $\mf w_i$ is skewered by $\gamma$.  Then for each $i$, we can choose a component $\mf w_i^+$ of 
$Z-\mf 
w_i$ so that $\gamma\subseteq\bigcap_i\mf w_i^+$.  Hence $\gamma$ is a $\langle 
g\rangle$--invariant embedded path in $X$ which is disjoint from $\mf w$, so $g^-,g^+$ lie on the same side of 
$\L\mf w$ in $\partial_\infty Z$, contradicting that $g$ skewers $\mf{w}$.

In conclusion, $g$ skewers a hyperplane $\mf{w}_i\in\mscr{W}(Z)$ and $\L\mf 
w_i\subseteq \L\mf w$. Part~(2) of the proposition now follows from part~(2) of Lemma~\ref{removing SP hyp}.
%
%
%
\end{proof}

\section{Bending hyperplanes.}\label{sec:bending}

\subsection{Controlling families of hyperplanes.}

Let a group $G$ act properly and cocompactly on a $\CAT$ cube complex $X$. As usual, we endow $X$ with its $\ell_1$ metric and set $d=\dim X$.


We are interested in the case when the quotient $G\backslash X$ is a \emph{special} cube complex, in the sense of~\cite[Definition 3.2]{Haglund-Wise-GAFA}. Recall that two distinct hyperplanes $\mf a_1,\mf a_2$ in $G\backslash X$ \emph{inter-osculate} if they both intersect and osculate: there is a square whose barycentre is in $\mf a_1\cap \mf a_2$, and there are also $1$--cubes dual to $\mf a_1,\mf a_2$ that share a vertex but do not lie in a common square. A special cube complex never has inter-osculating hyperplanes.

\begin{lem}\label{girthy cover 2}
Suppose that the quotient $G\backslash X$ is a special cube complex. 
Given disjoint hyperplanes $\mf{w}_1$ and $\mf{w}_2$, there exists a finite-index subgroup $H\leq G$ such that any two elements of $H\cdot\mf{w}_1\cup H\cdot\mf{w}_2$ are disjoint and at distance $\geq\tfrac{1}{d}\cdot d(\mf{w}_1,\mf{w}_2)$.
\end{lem}
\begin{proof}
By Corollary~7.9 in \cite{Haglund-Wise-GAFA}, halfspace-stabilisers are separable in $G$.  Thus, Lemma~\ref{girthy cover} 
allows us to assume that any two hyperplanes in the same $G$--orbit are disjoint and at distance $\geq\tfrac{1}{d}\cdot 
d(\mf{w}_1,\mf{w}_2)$. Let $n\geq 0$ be maximal such that $\mscr{W}(\mf{w}_1|\mf{w}_2)$ contains $n$ pairwise-disjoint 
hyperplanes. By Dilworth's lemma, we have $n\geq\tfrac{1}{d}\cdot (d(\mf{w}_1,\mf{w}_2)-1)$.

\smallskip
{\bf Claim:} \emph{Given disjoint $\mf{u}_1,\mf{u}_2\in\mscr{W}(X)$ and $m\geq 0$ maximal such that 
$\mscr{W}(\mf{u}_1|\mf{u}_2)$ contains $m$ pairwise-disjoint hyperplanes, there exists a finite-index subgroup $L\leq G$ such 
that every element of $L\cdot\mf{u}_2$ is separated from $\mf{u}_1$ by at least $m$ pairwise-disjoint hyperplanes (just 
disjoint from $\mf{u}_1$ if $m=0$).}

\smallskip
Applying the claim to $\mf{u}_i=\mf{w}_i$ clearly concludes the proof. We will prove the claim by induction on $m\geq 0$.

The base step $m=0$ is immediate taking $L=G$, as the quotient $G\backslash X$ has no inter-osculating 
hyperplanes.  

When 
$m\geq 1$, we can pick $\mf{u}\in\mscr{W}(\mf{u}_1|\mf{u}_2)$ such that $\mscr{W}(\mf{u}|\mf{u}_1)$ contains $m-1$ 
pairwise-disjoint hyperplanes. The inductive hypothesis yields a finite-index subgroup $K\leq G$ such that every element of 
$K\cdot\mf{u}$ is separated from $\mf{u}_1$ by at least $m-1$ pairwise-disjoint hyperplanes and such that no element of 
$K\cdot\mf{u}$ is transverse to $\mf{u}_2$ or $\mf{u}_1$.

Since no two elements of $K\cdot\mf{u}$ are transverse, the corresponding restriction quotient (see~\cite[Section 2.3]{CS}) 
of $X$ is a tree $\mc{T}$.  Recall that the preimage in $\mc{T}$ of the midpoint of any edge is a hyperplane 
in 
$K\cdot \mf{u}$, every hyperplane in $K\cdot\mf{u}$ is sent to the midpoint of an edge, and all other 
hyperplanes are collapsed to vertices.

Since the hyperplanes $\mf{u}_1$ and $\mf{u}_2$ are contained in distinct connected components of 
$X-\bigcup K\cdot\mf{u}$, they get collapsed to distinct vertices $v_1,v_2\in\mc{T}$. 

Let $k\in K$ be such that $kv_2\neq v_1$.  Then $kv_2$ and $v_1$ are separated by an edge of $\mc{T}$; the 
preimage of the midpoint of this edge is a hyperplane $\mf{v}\in K\cdot\mf{u}$ that separates the preimages of 
$kv_2$ and $v_1$.  In particular, $\mf{v}$ separates $\mf{u}_1$ from $k\mf{u}_2$.
Hence, the set
\[\mscr{W}(k\mf{u}_2|\mf{u}_1)\supseteq \{\mf{v}\}\sqcup\mscr{W}(\mf{v}|\mf{u}_1)\]
contains at least $m$ pairwise-disjoint hyperplanes. The $K$--stabiliser of $v_2$ is separable in $K$ by Corollary~7.9 in 
\cite{Haglund-Wise-GAFA}.  It follows that there exists a finite-index subgroup $L\leq K$ such that $v_1\not\in L\cdot v_2$. 
Every element of $L\cdot\mf{u}_2$ is then separated from $\mf{u}_1$ by at least $m$ pairwise-disjoint hyperplanes. 
\end{proof}

\begin{prop}\label{constructing systems of switches}
Suppose that $G$ is one-ended and that $X$ is essential, hy\-per\-plane-essential and $\delta$--hyperbolic.  For every $n>0$, 
there exist $m\geq 4$, hyperplanes $\mf{u}_1,\dots,\mf{u}_m$ and a finite-index subgroup $H\leq G$ such that:
\begin{itemize}
\item $G\cdot\{\mf{u}_1,\dots,\mf{u}_m\}=\mscr{W}(X)$; 
\item $\mf{u}_i$ is transverse to $\mf{u}_{i+1}$ for every $1\leq i< m$;
\item any two elements of $H\cdot\mf{u}_{i-1}\cup H\cdot\mf{u}_{i+1}$ are at distance $\geq n$, for every $1<i<m$;
\item $H\cdot\mf{u}_i\neq H\cdot\mf{u}_j$ whenever $i\neq j$.
\end{itemize}
\end{prop}
\begin{proof}
By Theorem~1.1 in \cite{Agol}, we can assume that the quotient $G\backslash X$ is a special cube complex.  By 
Lemma~\ref{girthy cover}, we can further assume that any two hyperplanes in the same $G$--orbit are disjoint and at distance 
${\geq n}$. By Lemma~\ref{not one-ended}, there exists a sequence $\mf{u}_1,\dots,\mf{u}_m$ of (not necessarily distinct)
hyperplanes satisfying the first two conditions. 

Since hyperplane-stabilisers are separable~\cite[Theorem 1.3]{Haglund-Wise-GAFA}, the fourth condition can always be ensured 
by passing to a further finite-index subgroup,  as long as the other three conditions are satisfied and the $\mf{u}_i$ are 
pairwise distinct. We will progressively modify $\mf{u}_1,\dots,\mf{u}_m$ in order to ensure that we are in this situation.
  
Consider $1<j<m$ and a finite index-subgroup $H\leq G$ such that $\mf{u}_1,\dots,\mf{u}_j$ are pairwise distinct and the third 
condition holds for all ${i<j}$.  Since the action $G_{\mf{u}_j}\acts\mf{u}_j$ is proper and cocompact by Lemma~\ref{cocompact 
hyperplanes}, Lem\-ma~\ref{disjoining hyperplanes} yields $g\in G_{\mf{u}_j}$ such that $d(g\mf{u}_{j+1},\mf{u}_{j-1})\geq dn$. 
We can moreover ensure that $g\mf{u}_{j+1}\not\in\{\mf{u}_1,\dots,\mf{u}_j\}$. For $l\geq j+1$, we replace each $\mf{u}_l$ with 
$g\mf{u}_l$. Note that the new hyperplanes still satisfy the first two conditions and, for $i<j$, also the third. Since now 
$d(\mf{u}_{j-1},\mf{u}_{j+1})\geq dn$, Lemma~\ref{girthy cover 2} yields a finite-index subgroup $K\leq H$ such that any two 
elements of $K\cdot\mf{u}_{j-1}\cup K\cdot\mf{u}_{j+1}$ are at distance $\geq n$. Replacing $H$ with $K$, the third condition 
is now satisfied for $i\leq j$ and $\mf{u}_1,\dots,\mf{u}_{j+1}$ are pairwise distinct.
\end{proof}

Now let $X$ be $\delta$--hyperbolic and consider hyperplanes $\mf{v}_0,\dots,\mf{v}_m$ such that:
\begin{itemize}
\item $\mf{v}_i$ is transverse to $\mf{v}_{i+1}$ for every $0\leq i<m$;
\item $d(\mf{v}_{i-1},\mf{v}_{i+1})\geq n$ for every $0<i<m$.
\end{itemize}
If $x\in\mf{v}_0$ and $y\in\mf{v}_m$ are vertices of the respective hyperplanes, we set $x_0=x$ and, for each $0\leq i\leq m-1$, we define inductively $x_{i+1}$ as the gate-projection of $x_i$ to $\mf{v}_{i+1}$. Note that $x_{i+1}\in\mf{v}_i\cap\mf{v}_{i+1}$. Finally, set $x_{m+1}=y$. Joining each $x_i$ to $x_{i+1}$ by an $\ell_1$ geodesic, we obtain a path $\eta\cu\mf{v}_0\cup\dots\cup\mf{v}_m\cu X$, which we will refer to as a \emph{standard path} from $x$ to $y$. 

By construction, the segment of $\eta$ that joins $x_i$ to $x_{i+2}$ is a geodesic for all $0\leq i\leq m-1$. It follows that $\eta$ is a $k$--local geodesic, with:
\[k\geq\min_{1\leq i\leq m-1}d(x_i,x_{i+1})\geq\min_{1\leq i\leq m-1}d(\mf{v}_{i-1},\mf{v}_{i+1})\geq n.\]
When $X$ is $\delta$--hyperbolic and $n>8\delta$, Theorem~III.H.1.13 in \cite{BH} guarantees that $\eta$ is a 
$(3,2\delta)$--quasi-geodesic.  By the Morse Lemma, there exists a constant $K=K(\delta)$ such that every geodesic in $X$ 
connecting $x$ and $y$ is at Hausdorff distance $\leq K$ from any standard path $\eta$. 

\begin{lem}\label{sequences of hyperplanes}
Let $X$ and $\mf{v}_0,\dots,\mf{v}_m$ be as above, with $m\geq 2$ and $n>8\delta$. 
\begin{enumerate}
\item Any geodesic from a point $x\in\mf{v}_0$ to a point $y\in\mf{v}_m$ has Hausdorff distance $\leq K$ from any standard 
path $\eta\cu\mf{v}_0\cup\dots\cup\mf{v}_m$ from $x$ to $y$. In particular, the union $\mf{v}_0\cup\dots\cup\mf{v}_m$ is 
$K$--quasiconvex in $X$.
\item We have $d(\mf{v}_0,\mf{v}_m)\geq\tfrac{n(m-1)}{3}-2\delta$. Thus, $\mf{v}_0$ and $\mf{v}_m$ are disjoint.
\end{enumerate}
\end{lem}
\begin{proof}
We have already shown part~(1) in the above discussion.  Regarding part~(2), pick vertices $x\in\mf{v}_0$ 
and $y\in\mf{v}_m$ with $d(x,y)=d(\mf{v}_0,\mf{v}_m)$.
%

Let 
$\eta$ be a standard path from $x$ to $y$. As shown above, $\eta$ is an $n$--local geodesic and it contains points 
$x_1,\dots,x_m$ satisfying $d(x_i,x_{i+1})\geq n$. It follows that the domain of $\eta$ has length $\geq n(m-1)$ and we know 
that $\eta$ is a $(3,2\delta)$--quasi-geodesic. Thus:
\[d(\mf{v}_0,\mf{v}_m)=d(x,y)\geq\tfrac{n(m-1)}{3}-2\delta>0,\] as required.
\end{proof}

\subsection{Systems of switches.}\label{systems of switches sect}
Let $G$ be a one-ended hyperbolic group. We consider a proper cocompact $G$--action on an essential, hyperplane-essential, 
$\delta$--hyperbolic $\CAT$ cube complex $X$. We fix $M\geq 1$ such that, for every $\mf{w}\in\mscr{W}(X)$ and every vertex 
$x\in\mf{w}$, the orbit $G_{\mf{w}}\cdot x$ is $M$--dense in $\mf{w}$, using Lemma~\ref{cocompact hyperplanes}.

Let us denote by $\trans(X)$ the collection of subsets ${\{\mf{u},\mf{v}\}\cu\mscr{W}(X)}$ such that $\mf{u}$ is transverse 
to $\mf{v}$.  Given a subset $\mc{S}\cu\trans(X)$ and $\mf{u}\in\mscr{W}(X)$, let $\mc{S}_{\mf{u}}\cu\mscr{W}(X)$ be the collection 
of all those $\mf{v}$ with $\{\mf{u},\mf{v}\}\in\mc{S}$. 

\begin{defn}
An \emph{$n$--system of switches} is a pair $\mscr{S}=(\mc{S},H)$, where: 
\begin{itemize}
\item $H\lhd G$ is a finite-index subgroup;
\item $\mc{S}\cu\trans(X)$ is an $H$--invariant subset; 
\item for every $\mf{u}\in\mscr{W}(X)$, any two elements of $\mc{S}_{\mf{u}}$ are at distance $\geq n$. 
\end{itemize}
The \emph{support} of $\mscr{S}$ is the set $\supp\mscr{S}:=\{\mf{u}\in\mscr{W}(X)\mid\mc{S}_{\mf{u}}\neq\emptyset\}$. We say that $\mscr{S}$ is \emph{full} if $G\cdot\supp\mscr{S}=\mscr{W}(X)$.
\end{defn}

\begin{lem}\label{constructing systems of switches 2}
For every $n>0$, there exists a full $n$--system of switches.
\end{lem}
\begin{proof}
Let $\mf{u}_1,\dots,\mf{u}_m$ and $H$ be as provided by Proposition~\ref{constructing systems of switches}; 
passing to a further finite-index subgroup, we can assume that $H$ is normal in $G$.
Let $\mc{S}$ be the union of the sets $H\cdot\{\mf{u}_i,\mf{u}_{i+1}\}$ for $1\leq i<m$. 
Then the pair $\mscr{S}=(\mc{S},H)$ is a full $n$--system of switches.

Indeed, we have $\mf{w}\in\mc{S}_{\mf{u}_j}$ if and only if there exist $h\in H$ and $i$ satisfying ${h\cdot\{\mf{w},\mf{u}_j\}=\{\mf{u}_i,\mf{u}_{i+1}\}}$. Since $H\cdot\mf{u}_l\neq H\cdot\mf{u}_{l'}$ whenever $l\neq l'$, we must have $j\in\{i,i+1\}$ and $h\mf{u}_j=\mf{u}_j$. It follows that $\mc{S}_{\mf{u}_j}=H_{\mf{u}_j}\cdot\mf{u}_{j-1}\sqcup H_{\mf{u}_j}\cdot\mf{u}_{j+1}$, any two elements of which are at distance $\geq n$.
\end{proof}

We write $U(\mscr{S})\cu X$ for the union of all hyperplanes in $\supp\mscr{S}$, and $U_{\pitchfork}(\mscr{S})\cu X$ for the 
union of all intersections $\mf{u}\cap\mf{v}$ with $\{\mf{u},\mf{v}\}\in\mc{S}$. We have: 
\[U(\mscr{S})=U_{\pitchfork}(\mscr{S})\sqcup(U(\mscr{S})-U_{\pitchfork}(\mscr{S})),\] 
where every point of $U(\mscr{S})-U_{\pitchfork}(\mscr{S})$ lies in a unique hyperplane of $\supp\mscr{S}$, while points of $U_{\pitchfork}$ each lie in two distinct elements of $\supp\mscr{S}$. 

We will refer to $U_{\pitchfork}(\mscr{S})$ and $U(\mscr{S})-U_{\pitchfork}(\mscr{S})$, respectively, as the \emph{singular} and \emph{regular} part of $U(\mscr{S})$ (and to their points as \emph{singular} and \emph{regular points}). Note that every vertex of each hyperplane in $\supp\mscr{S}$ is a regular point.

We denote by $\comp\mscr{S}$ the set of all connected components of the regular part of $U(\mscr{S})$. Each element of $\comp\mscr{S}$
is a connected component of a set of the form $\mf{w}- \bigcup_{\mf{v}\in\mc{S}_{\mf{w}}}\mf{v}$ with $\mf{w}\in\supp\mscr{S}$. Instead, note that connected components of the singular part $U_{\pitchfork}(\mscr{S})$ are in one-to-one correspondence with elements of $\mc{S}$. For every regular point $x\in U(\mscr{S})$, we write $[x]$ for the only element of $\comp\mscr{S}$ that 
contains $x$. 

To each $n$--system of switches $\mscr{S}=(\mc{S},H)$ we can associate a bipartite graph $\mc{G}(\mscr{S})$ equipped with a 
cocompact $H$--action.  The vertex set is the disjoint union $\mc{S}\sqcup\comp\mscr{S}$. 
We join each vertex 
$\{\mf{u},\mf{v}\}\in\mc{S}$ to the four elements of $\comp\mscr{S}$ that contain $\mf{u}\cap\mf{v}$ in their closure; two 
are contained in $\mf{u}$ and two in $\mf{v}$. We call vertices of $\mc{G}(\mscr{S})$ \emph{regular} if they 
originate from 
elements of $\comp\mscr{S}$ and \emph{singular} if they originate from elements of $\mc{S}$.

All singular vertices have degree $4$ in $\mc{G}(\mscr{S})$, while regular vertices can have infinite degree 
in general.

Every hyperplane $\mf{w}\in\supp\mscr{S}$ determines a subgraph $\mc{G}(\mf{w})\cu\mc{G}(\mscr{S})$ spanned by 
the elements of $\comp\mscr{S}$ that are contained in $\mf{w}$. Note that $\mc{G}(\mf{w})$ is naturally isomorphic to the 
barycentric subdivision of the restriction quotient of $X$ corresponding to the set 
$\mc{S}_{\mf{w}}\cu\mscr{W}(X)$.  In particular, since the definition of a system of switches 
means that $\mc{S}_{\mf{w}}$ consists of pairwise disjoint hyperplanes, the graph $\mc{G}(\mf{w})$ is a tree.

\begin{ex}
It can be helpful to have the following special case in mind in this subsection. Let $G$ be the fundamental group of a closed, orientable surface $\Sigma$. Let $X$ be a $2$--dimensional cubulation constructed by applying Sageev's construction to a suitable finite filling collection of curves $\mc{F}$ on $\Sigma$.

A choice of a system of switches $\mscr{S}$ corresponds to a choice of a finite cover $\Sigma'\ra\Sigma$ and a set of intersections between the lifts of $\mc{F}$ to $\Sigma'$. We will perform surgery on these intersections and obtain a new family of curves in $\Sigma'$ (sometimes, one of the new curves will fill $\Sigma'$). Here, the main difference from the general case is that regular vertices of $\mc{G}(\mscr{S})$ always have degree $2$.
\end{ex}

\begin{defn}
We say that a connected subgraph $A\cu\mc{G}(\mscr{S})$ is:
\begin{enumerate}
\item \emph{two-sided} if every singular vertex that lies in $A$ has degree $2$ in $A$;
\item \emph{star-complete} if, whenever a regular vertex lies in $A$, its star in $\mc{G}(\mscr{S})$ is also 
contained in $A$.
\end{enumerate}
\end{defn}

Given a connected subgraph $A\cu\mc{G}(\mscr{S})$, let $U(A)\cu U(\mscr{S})$ be obtained by taking the union of the closures 
of the elements of $\comp\mscr{S}$ that lie in $A$.  

We say that a subspace $Y\subset X$ is \emph{locally codimension-1} if for all $y\in Y$, there is a 
neighbourhood $N$ of $y$ in $X$ equipped with a homeomorphism $(Y\cap N)\times[-\frac12,\frac12]\to N$ 
restricting to the inclusion $Y\cap N\to N$ on $(Y\cap N)\times\{0\}$.

\begin{prop}\label{first properties}
There exists $K=K(\delta)$ such that all the following hold for $n>8\delta$ and any $n$--system of switches $\mscr{S}=(\mc{S},H)$.
\begin{enumerate}
\item The graph $\mc{G}(\mscr{S})$ is a forest.
\item Consider $\mf{u},\mf{v}\in\supp\mscr{S}$ such that $\mc{G}(\mf{u})$ and $\mc{G}(\mf{v})$ are contained in 
the same connected component of $\mc{G}(\mscr{S})$. Then $\mf{u}$ and $\mf{v}$ are transverse if and only if 
$\{\mf{u},\mf{v}\}\in\mc{S}$.
\item For every connected subgraph $A\cu\mc{G}(\mscr{S})$, the subset $U(A)\cu U(\mscr{S})$ is connected and 
$K$--quasiconvex.
\item A connected subgraph $A\cu\mc{G}(\mscr{S})$ is star-complete and two-sided if and only if the subset 
$U(A)\cu X$ is locally codimension-1.
\item The orbit $G\cdot U(A)$ is locally finite in $X$ if and only if no regular vertex of $\mc{G}(\mscr{S})$ 
lies in infinitely many pairwise-distinct $H$--translates of $A$.
\end{enumerate} 
\end{prop}
\begin{proof}
Consider an immersed path $\g\cu\mc{G}(\mscr{S})$ between two regular vertices $[x],[y]\in\comp\mscr{S}$. Let 
$\mf{v}_0,\dots,\mf{v}_k\in\supp\mscr{S}$ be the hyperplanes containing the elements of $\g\cap\comp\mscr{S}$, 
in the same order as they appear moving from $[x]$ to $[y]$ along $\g$. In particular, $x\in [x]\cu\mf{v}_0$, 
$y\in [y]\cu\mf{v}_k$, the hyperplane $\mf{v}_i$ is transverse to $\mf{v}_{i+1}$ for each $0\leq i\leq k-1$ and 
we have $d(\mf{v}_{i-1},\mf{v}_{i+1})\geq n$ for every $0<i<k$.

It follows that $\mf{v}_0,\dots,\mf{v}_k$ satisfy the hypotheses of Lemma~\ref{sequences of hyperplanes}. Denoting by $\tilde\gamma\cu U(\mscr{S})\cu X$ any standard path joining $x$ and $y$, any geodesic joining $x$ and $y$ in $X$ is at Hausdorff distance $\leq K$ from $\tilde\g$. This shows part~(3), while parts~(1) and~(2) follow from part~(2) of Lemma~\ref{sequences of hyperplanes}. Part~(4) is obvious.

Finally, we prove part~(5). Observe that $G\cdot U(A)$ is locally finite if and only if $H\cdot U(A)$ is. This fails if and only if a point $z\in U(\mscr{S})$ lies in infinitely many pairwise-distinct $H$--translates of $U(A)$, say $h_iU(A)=U(h_iA)$. Without loss of generality, $z$ is a vertex of a hyperplane of $X$. The above then happens if and only if the vertex $[z]\in\mc{G}(\mscr{S})$ lies in the pairwise-distinct subgraphs $h_iA$. 
\end{proof}

\begin{defn}
A \emph{crooked hyperplane} is a connected, two-sided, star-com\-plete subtree $\G\cu\mc{G}(\mscr{S})$ such 
that no regular vertex of $\mc{G}(\mscr{S})$ lies in infinitely many pairwise-distinct $H$--translates of $\G$.
\end{defn}

Given a crooked hyperplane $\G\cu\mc{G}(\mscr{S})$, we denote by $G_{\G}$ the $G$--stabiliser of the subset\footnote{The 
$G$--stabiliser of $\G$ itself would not make sense, as, although $H$ acts on $\mc{G}(\mscr{S})$, the group $G$ does not.} $U(\G)\cu X$. Part~(5) of 
Proposition~\ref{first properties} and Lemma~2.3 in \cite{Hagen-Susse} show that the action $G_{\G}\acts U(\G)$ is proper and 
cocompact. By part~(3), the subgroup $G_{\G}\leq G$ is quasiconvex.  Moreover, from part~(4) and Mayer--Vietoris, we see 
that $X- U(\G)$ has exactly two connected components. We write $G_{\G}^0\leq G_{\G}$ for the subgroup (of index at 
most two) that leaves invariant both connected components of $X- U(\G)$.

\begin{rmk}\label{boundaries of crooked hyperplanes}
For every crooked hyperplane $\G\cu\mc{G}(\mscr{S})$, there is a natural map $\iota\colon\partial_{\infty}\G\ra\partial_{\infty}U(\G)$ taking the endpoint of a ray $\g$ in the tree $\G$ to the endpoint at infinity of any standard path $\tilde\g\cu U(\G)$. The map $\iota$ is a homeomorphism onto its image, and it satisfies:
\[\partial_{\infty}U(\G)=\iota(\partial_{\infty}\G)\sqcup\bigcup_{c\in\G\cap\comp\mscr{S}}\partial_{\infty}c.\]
Thus, $\partial_{\infty}U(\G)$ is nonempty even when each element of $\comp\mscr{S}$ is bounded.
\end{rmk}

\begin{prop}\label{existence}
Let $\mscr{S}$ be an $n$--system of switches. Every compact, two-sided subtree $A\cu\mc{G}(\mscr{S})$ is 
contained in a crooked hyperplane.
\end{prop}
\begin{proof}
Let us write $\mc{G}=\mc{G}(\mscr{S})$ for simplicity and let $\mc{G}_2\cu\mc{G}$ be the subset of singular 
vertices; recall that the action $H\acts\mc{G}$ is cocompact. Since the action $H\acts\mc{G}$ has separable 
vertex- and edge-stabilisers by Corollary~7.9 in \cite{Haglund-Wise-GAFA}, there exists a finite-index subgroup 
$L\lhd H$ such that $A$ projects injectively to the quotient $\overline{\mc{G}}=L\backslash\mc{G}$ and such 
that every element of $\mc{G}_2$ projects to a degree-4 vertex of $\overline{\mc{G}}$. 
Let $\pi\colon\mc{G}\ra\overline{\mc{G}}$ denote the quotient projection.

For each $v\in\pi(\mc{G}_2)$, we choose a partition of the four edges incident to $v$ into two pairs. We do so ensuring that, if $v\in\pi(A)$, one element of the partition consists precisely of the two edges lying in $\pi(A)$. We now lift these partitions to $\mc{G}$. Let $\mc{G}'$ be the graph obtained from $\mc{G}$ by replacing every vertex in $\mc{G}_2$ with two vertices of degree $2$, according to the chosen partitions. The graph $\mc{G}'$ naturally comes equipped with an immersion $f\colon\mc{G}'\ra\mc{G}$.

By construction, there exists a connected component $\G\cu\mc{G}'$ such that $A\cu f(\G)$. It is clear that 
$f(\G)$ is a connected, two-sided, star-complete subtree of $\mc{G}$. Each regular vertex of $\mc{G}$ lies in 
at most one $L$--translate of $f(\G)$. It follows that every regular vertex of $\mc{G}$ lies in finitely many 
$H$--translates of $f(\G)$. This shows that $f(\G)$ is a crooked hyperplane, concluding the proof.
\end{proof}

If $\G\cu\mc{G}(\mscr{S})$ is a crooked hyperplane, we denote by $\mc{T}_{\G}$ the connected component of $\mc{G}(\mscr{S})$ that contains $\G$.

Let $D$ be as in Lemma~\ref{small projection}, let $K$ be as in Proposition~\ref{first properties} and let $M$ be the constant chosen at the beginning of this subsection.

\begin{prop}\label{skewering new}
Consider $n>2(M+D+2K)$ and a full $n$--system of switches $\mscr{S}=(\mc{S},H)$.  For every non-constant geodesic $\g\cu X$, there exists a regular vertex $v\in\mc{G}(\mscr{S})$ with the following property. For every crooked hyperplane 
$\G\cu\mc{G}(\mscr{S})$ intersecting the orbit $H\cdot v$, there exists $g\in G$ such that $g\g$ intersects $U(\mc{T}_{\G})$ in a single point, which lies in $U(\G)$.
\end{prop}

Before proving Proposition~\ref{skewering new}, we need to obtain a couple of lemmas. 

\begin{lem}\label{most stuff is far}
Let $\mc{T}$ be a connected component of $\mc{G}(\mscr{S})$. Consider a regular point $p\in U(\mc{T})$ and let $\mf{w}$ 
be the hyperplane containing $[p]\in\comp\mscr{S}$. Then, for every $x\in U(\mc{T})-[p]$, we have 
$d(p,\pi_{\mf{w}}(x))\geq d(p,\mf{w}-[p])-2K$.
\end{lem}
\begin{proof}
Let $\tilde\g$ be a standard path from $p$ to $x$ and let $\alpha$ be an $\ell_1$ geodesic joining $p$ and $x$.  Let 
$y\in\tilde\g$ be the first point that does not lie in $[p]$ and let $y'\in\alpha$ be a point that is closest to $y$; in 
particular, we have $d(y,y')\leq K$. It follows that:
\[(x\cdot y)_p\geq (x\cdot y')_p-K=d(p,y')-K\geq d(p,y)-2K\geq d(p,\mf{w}-[p])-2K.\]
Since $p,y\in\mf{w}$, every element of $\mscr{W}(p|x,y)$ is transverse to $\mf{w}$;  hence we have 
$\mscr{W}(p|x,y)\cu\mscr{W}(p|\pi_{\mf{w}}(x))$. As $(x\cdot y)_p=\#\mscr{W}(p|x,y)$, this concludes the proof.
\end{proof}

\begin{lem}\label{large components}
Let a hyperplane $\mf{w}\in\supp\mscr{S}$ and a vertex $x\in\mf{w}$ be given.  If $n>2M$, there exists $g\in G_{\mf{w}}$ such 
that $d(gx,\mf{w}-[gx])\geq\tfrac{n}{2}-M$.
\end{lem}
\begin{proof}
Since $\mc{S}$ is $H_{\mf{w}}$--invariant and the action $H_{\mf{w}}\acts\mf{w}$ is essential, 
the component $[x]\cu\mf{w}$ must have at least two boundary components. Given that any two 
boundary components of $[x]$ are at distance $\geq n$ from each other and $[x]$ is connected, there exists a vertex 
$q\in\mf{w}$ such that $[q]=[x]$ and $d(q,\mf{w}-[q])\geq\tfrac{n}{2}$. By the definition of $M$, there exists an 
element $g\in G_{\mf{w}}$ such that $d(gx,q)\leq M<\tfrac{n}{2}$. Hence $[gx]=[q]$ and we obtain $d(gx,\mf{w}-[gx])\geq 
d(q,\mf{w}-[q])-M\geq\tfrac{n}{2}-M$.
\end{proof}

\begin{proof}[Proof of Proposition~\ref{skewering new}]
By Lemma~\ref{small projection} there exists a hyperplane $\mf{w}\in\mscr{W}(\g)$ for which $\diam\pi_{\mf{w}}(\g)\leq D$. 
Since $\mscr{S}$ is full, we can replace $\g$ with a $G$--translate and assume that $\mf{w}\in\supp\mscr{S}$. 
Let $x$ be the point of intersection between $\g$ and $\mf{w}$. Let $\mc{T}\cu\mc{G}(\mscr{S})$ be 
the connected component that contains the vertex $[x]\in\mc{G}(\mscr{S})$.
Up to replacing $\g$ with a $G_{\mf{w}}$--translate, 
Lemma~\ref{large components} allows us to assume that $d(x,\mf{w}-[x])>2K+D$.  
Then, Lemma~\ref{most stuff is far} gives:
\[d(x,\pi_{\mf{w}}(U(\mc{T})-[x]))>D\geq\diam\pi_{\mf{w}}(\g).\] 
We conclude that $\g$ and $U(\mc{T})-[x]$ are disjoint.

If $\G\cu\mc{G}(\mscr{S})$ is a crooked hyperplane containing $[x]$, we have $\mc{T}_{\G}=\mc{T}$ and:
\[\g\cap U(\mc{T}_{\G})=\g\cap[x]=\{x\}.\] 

Finally, set $v=[x]$. Let $\G'$ be a crooked hyperplane containing $hv$ for some $h\in H$. Then $h\g$ intersects $U(\mc{T}_{\G'})$ in the single point $hx\in U(\G')$.
\end{proof}

We say that $g\in G$ \emph{skewers} a crooked hyperplane $\G$ if we have $g\overline C\subseteq C$ for one of 
the two connected components $C\cu X- U(\G)$, where $\overline C$ denotes the closure. In this case, we have 
$d(gC,U(\G))>0$, by cocompactness of $G_{\G}^0\acts U(\G)$. We remark that, if $\G$ is skewered by an element 
of $G$, the subgroup $G_{\G}^0\leq G$ is codimension-one.

\begin{cor}\label{skewering cor}
Consider $n>2(M+D+2K)$, a full $n$--system of switches $\mscr{S}=(\mc{S},H)$ and a crooked hyperplane 
$\G\cu\mc{G}(\mscr{S})$ projecting surjectively to $H\backslash\mc{G}(\mscr{S})$. Then:
\begin{enumerate}
\item every non-constant geodesic $\g\cu X$ has a $G$--translate intersecting $U(\G)$ is a single point;
\item for every infinite-order element $g\in G$, a $G$--conjugate of a power of $g$ skewers $U(\G)$.
\end{enumerate}
\end{cor}
\begin{proof}
Part~(1) is immediate from Proposition~\ref{skewering new} and the fact that $U(\G)$ is contained in $U(\mc{T}_{\G})$. 

Regarding part~(2), we can replace $g$ with a power and assume that $g$ admits an axis $\g\cu X$ \cite{Haglund}. Again by Proposition~\ref{skewering new}, we can replace $g$ with a conjugate and assume that $\g$ intersects $U(\mc{T}_{\G})$ in a single point $x\in U(\G)$. It follows that one connected component $C_+\cu X- U(\G)$ contains the positive half of $\g-\{x\}$, while 
the other component $C_-\cu X- U(\G)$ contains the negative half. 

Let us pick $m>0$ such that $g^m\in H$. Note that $g^mU(\G)$ and $U(\G)$ are disjoint. Otherwise, we would have:
\[\emptyset\neq g^mU(\G)\cap U(\G)=U(g^m\G\cap\G).\] 
Hence $g^m\mc{T}_{\G}=\mc{T}_{\G}$ and
\[g^mx\in g^mU(\mc{T}_{\G})=U(\mc{T}_{\G}),\] 
contradicting the fact that $\g\cap U(\mc{T}_{\G})=\{x\}$.  

Now, since $g^mU(\G)$ and $U(\G)$ are disjoint, we have $g^mU(\G)\cu C_+$. Note moreover that $g^mC_-\cap C_-\neq\emptyset$, as both sets contain a sub-ray of $\g$. We conclude that $g^m\overline{C_+}\subseteq C_+$, i.e.\ $g^m$ skewers $\G$.
\end{proof}

Part~(2) of Corollary~\ref{skewering cor} and Proposition~\ref{cubulating with abstract hyperplanes} now yield:

\begin{cor}\label{main}
Consider $n>2(M+D+2K)$, a full $n$--system of switches $\mscr{S}=(\mc{S},H)$ and a crooked hyperplane 
$\G\cu\mc{G}(\mscr{S})$ projecting surjectively to $H\backslash\mc{G}(\mscr{S})$. There exists an essential $\CAT$ cube 
complex $X_{\G}$ and a proper cocompact action $G\acts X_{\G}$ with a single orbit of hyperplanes. All hyperplane-stabilisers 
of $G\acts X_{\G}$ are conjugate to $G_{\G}\leq G$.
\end{cor}

The following proves Theorem~\ref{one orbit intro} in the one-ended case.

\begin{cor}\label{cor1}
Every cocompactly cubulated one-ended hyperbolic group admits an essential, cocompact cubulation with a single orbit of hyperplanes.
\end{cor}
\begin{proof}
Let $G$ be a one-ended hyperbolic group with a proper cocompact action on a $\CAT$ cube complex $X$.  By 
Proposition~\ref{facts about panel collapse}, we can assume that $X$ is essential and hyperplane-essential. 
Lemma~\ref{constructing systems of switches 2} provides a full $n$--system of switches $\mscr{S}$ with $n>2(M+D+2K)$, where 
$D,K,M$ are as above. Every connected component $\mc{T}\cu\mc{G}(\mscr{S})$ is a tree that projects surjectively to 
$H\backslash\mc{G}(\mscr{S})$. Since $X$ is hyperplane-essential, $\mc{T}$ has no leaves. The stabiliser $H_{\mc{T}}$ acts 
cocompactly on $\mc{T}$, hence minimally. It follows that there exists a compact, two-sided subtree 
$A\cu\mc{T}\cu\mc{G}(\mscr{S})$ that projects surjectively to the finite graph $H\backslash\mc{G}(\mscr{S})$. By 
Proposition~\ref{existence}, there exists a crooked hyperplane $\G\cu\mc{G}(\mscr{S})$ containing $A$ and we can apply 
Corollary~\ref{main}.
\end{proof}

\begin{rmk}
The cubulation provided by Corollary~\ref{main} is in general not hyperplane-essential. This is due to the following configuration,
in which we may find two crooked hyperplanes $\mc{C}_1=U(\G)$ and $\mc{C}_2=gU(\G)$. Denoting by $\mc{C}_i^{\pm}$ the
two connected components of $X-\mc{C}_i$, all four intersections
$\mc{C}_1^{\pm}\cap\mc{C}_2^{\pm}$ might contain points arbitrarily far from $\mc{C}_1\cup\mc{C}_2$, even if, 
say, the intersection 
$\mc{C}_1\cap\mc{C}_2^+$ is bounded. In this case, the cubulation of $G$ arising from $G\cdot U(\G)$ has transverse
hyperplanes $\mf{w}_1,\mf{w}_2$ arising from $\mc{C}_1,\mc{C}_2$, 
but it is impossible to skewer $\mf{w}_1\cap\mf{w}_2$ with a hyperbolic element stabilising $\mf{w}_1$.
\end{rmk}

\subsection{The infinitely-ended case.}\label{infinitely-ended sect}

In this subsection, we complete the proof of Theorem~\ref{one orbit intro} by addressing the case where $G$ is not one-ended. 

The idea is to construct ``antennae'' (in the sense of \cite{Wise-antenna}) in the maximal Bass-Serre tree and to attach 
crooked hyperplanes constructed for the one-ended vertex groups.  We now describe the construction in detail.

Let $G$ be a cocompactly cubulated hyperbolic group. By \cite{Dunwoody1,Dunwoody2},  $G$ is the fundamental group of a 
finite graph of groups $\mscr{G}$ where edge groups are finite and vertex groups are either finite or one-ended. Let 
$G\acts\mc{T}$ be the action on the corresponding Bass-Serre tree. Let $G_1,\dots,G_k$ be the one-ended vertex groups and let 
$v_1,\dots,v_k$ be their fixed vertices in $\mc{T}$. 

\subsubsection{The orbicomplex $\overline X$}\label{subsubsec:orbicomplex}
By Proposition~1.2 in \cite{Bow-JSJ}, each $G_i$ is quasiconvex in $G$ and, by Theorem~H in \cite{Haglund-GD}, each $G_i$ is 
cocompactly cubulated.  By Proposition~\ref{facts about panel collapse}, we can pick a proper cocompact action of each $G_i$ 
on an essential, hyperplane-essential $\CAT$ cube complex $X_i$. Cubically subdividing if necessary, we can assume that each 
$G_i\acts X_i$ has no hyperplane inversions. Each finite subgroup $F\leq G_i$ has a global fixed point in $X_i$; hence $F$ 
preserves a cube of $X_i$ and, since there are no hyperplane inversions, $F$ must fix a vertex. 

We now construct a specific ``orbicomplex'' $\overline X$ with $G=\pi_1\overline X$.  We start with the disjoint union of the 
quotient orbicomplexes $\overline X_i:=G_i\backslash X_i$ for ${1\leq i\leq k}$, plus a singleton for every finite vertex 
group of $\mscr{G}$. For each edge of $\mscr{G}$, we add an edge connecting the corresponding orbicomplexes $\overline X_i$ 
or singletons. If $F$ is the associated edge group, we ensure that the attaching vertex in $\overline X_i$ is the projection 
of a vertex of $X_i$ that is fixed by the image of the homomorphism $F\ra G_i$. 

Let $G\acts X$ be the action on the universal cover of $\overline X$.  This is a proper cocompact action on an essential, 
hyperplane-essential $\CAT$ cube complex. 

We do not want to identify effectively parallel hyperplanes, as this can alter
the action $G\acts\mc{T}$. However, the construction of $X$ will also be required later on in the 
proof of Theorem~\ref{infinitely many intro}. For that purpose, we observe that Proposition~\ref{facts about panel collapse}
yields:

\begin{lem}\label{assembling cubulations}
Let $G$ be a cocompactly cubulated hyperbolic group and let $G_1,\dots,G_k$ be the one-ended factors of the maximal splitting 
of $G$ over finite subgroups. Given bald cubulations $G_i\acts X_i$, there exists a bald cubulation $G\acts X$ such that, 
for each $i$, the $G_i$--essential core of the restriction $G_i\acts X$ is $G_i$--equivariantly isomorphic to $X_i$.
\end{lem}


In fact, the case when $G$ has torsion will only be needed in 
Section~\ref{sec:bald_hyperbolic} when we prove Theorem~\ref{infinitely many intro}. 

In the remainder of this section, we can and shall assume that $\overline X=G\backslash X$ is a genuine cube 
complex by passing to a torsion-free finite-index subgroup of $G$ (whose existence is guaranteed by 
specialness~\cite{Agol,Haglund-Wise-GAFA}) and applying the following trick:

\begin{lem}\label{supergroups}
Let $G$ be a hyperbolic group with a finite-index subgroup $H\leq G$. Then, if $H$ admits a cocompact cubulation
with a single orbit of hyperplanes, so does $G$.
\end{lem}
\begin{proof}
Let $H\acts X$ be a cocompact cubulation with a single $H$--orbit of hyperplanes. 
We pick a hyperplane $\mf{w}\in\mscr{W}(X)$, with associated halfspaces $\mf{h}^+,\mf{h}^-$ and halfspace-stabiliser 
$H^0_{\mf{w}}\leq H$. Consider the two sets $\mf{H}^{\pm}:=\L\mf{h}^{\pm}- \L H^0_{\mf{w}}$, where limit sets are 
taken in $\partial_{\infty}H=\partial_{\infty}G$. Since the action $H\acts X$ is necessarily essential, the sets $\mf{H}^{\pm}$ 
are nonempty and we obtain an abstract hyperplane $(H^0_{\mf{w}},\mf{H})$ for both $H$ and $G$.

We now apply Proposition~\ref{cubulating with abstract hyperplanes} to the collection $\mc{H}=G\cdot\mf{H}$. 
We obtain a cocompact action $G\acts X(\mc{H})$ with a single $G$--orbit of hyperplanes. 
If $g\in G$ has infinite-order, a power of $g$ lies in $H$, where it skewers a hyperplane of $X$. 
It follows that the points $g^{\pm}\in\partial_{\infty}G=\partial_{\infty}H$ 
are separated by an abstract hyperplane in $H\cdot\mf{H}\cu\mc{H}$. 
This shows that the action $G\acts X(\mc{H})$ is proper and, thus, the desired cubulation of 
$G$.
\end{proof}

We thus assume, in the rest of the discussion, that $G$ is {\bf torsion-free}. 

The $\CAT$ cube complex $X$ constructed right before Lemma~\ref{assembling cubulations} 
comes equipped with a natural $G$--equivariant projection $\pi\colon X\ra\mc{T}$.  For every open edge 
$e\cu\mc{T}$, the preimage $\pi^{-1}(e)\cu X$ consists of a single separating (open) edge of $X$. For every vertex $v\in\mc{T}$, the 
preimage $X_v:=\pi^{-1}(v)\cu X$ is a convex subcomplex of $X$ with a proper, cocompact, essential, hyperplane-essential 
action $G_v\acts X_v$; here $G_v\leq G$ denotes the stabiliser of the vertex $v\in\mc{T}$. We identify 
$X_{v_i}=\pi^{-1}(v_i)$ with $X_i$ and the action $G_{v_i}\acts X_{v_i}$ with $G_i\acts X_i$.

\subsubsection{Antennae}\label{subsubsec:antenna}

Recall that $G\acts X$ without inversions. Thus, by \cite{Haglund}, every $g\in G-\{1\}$ admits an axis $\g\cu X$. The projection $\pi(\g)\cu\mc{T}$ is either a single vertex (if $g$ is elliptic in $\mc{T}$), or an axis for $g$ in $\mc{T}$. 

Let $G_{BS}\cu G$ be the subset of elements that admit an axis $\g\cu X$ that projects \emph{injectively} to $\mc{T}$. In other words, this is an axis that does not contain any edges lying in one of the nontrivial fibres of the map $\pi\colon X\ra\mc{T}$.

Since $X$ is locally finite and the action $G\acts X$ is cocompact, we can find a finite collection $\mscr{P}$ of length-two paths in $\mc{T}$ with the following property. Every element of $G_{BS}$ has a conjugate whose axis (in $\mc{T}$) contains an element of $\mscr{P}$ as a sub-path.

Possibly replacing each element of $\mscr{P}$ by a $G$--translate, there exists a geodesic segment $\alpha_1\cu\mc{T}$ that intersects 
each element of $\mscr{P}$ in its middle vertex. Replacing each vertex $v_i\in\mc{T}$ with a $G$--translate if necessary, 
there exists another geodesic segment $\alpha_2\cu\mc{T}$ containing $v_1,\dots,v_k$ 
(recall that the $v_i$ are representatives of the $G$--orbits of vertices of $\mc{T}$ with infinite $G$--stabiliser, 
as chosen at the beginning of Section~\ref{infinitely-ended sect}). 
We can moreover assume that $\alpha_1$ 
and $\alpha_2$ intersect at an endpoint and only at that endpoint.

Let $\mc{A}\cu\mc{T}$ be the union of $\alpha_1$, $\alpha_2$ and all elements of $\mscr{P}$, shown in 
Figure~\ref{antenna1}.  This is an antenna with some missing arms (cf.\ Section~2.1 in \cite{Wise-antenna}). We also 
choose a finite tree $A\cu X$ with $\pi(A)=\mc{A}$.

\begin{figure}[h]
     \begin{overpic}[width=0.66\textwidth,keepaspectratio]{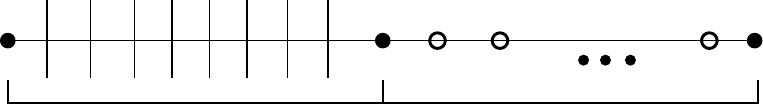}
          \put(20,-3){$\alpha_1$}
          \put(70,-3){$\alpha_2$}
          \put(22,13){$\mscr{P}$}
          \put(56,10){$v_1$}
          \put(92,10){$v_k$}
     \end{overpic}
\caption{The subtree $\mc{A}\cu\mc{T}$.}
\label{antenna1}
\end{figure}

%
%

\subsubsection{The cube complex $\overline U$}\label{subsubsec:orbicomplex_U}
As in the proof of Corollary~\ref{cor1}, there exist systems of switches $\mscr{S}_i$ in $X_i$ and crooked hyperplanes 
$\G_i\cu\mc{G}(\mscr{S}_i)$ that satisfy the hypotheses of Corollary~\ref{skewering cor}. Thus, every 
$g\in G_i-\{1\}$ has a conjugate of a power skewering 
$U(\G_i)\cu X_i$, and every geodesic in $X_i$ has a $G_i$--translate intersecting $U(\G_i)$ in a single point.

\begin{rmk}\label{annoying issue}
Replacing $\G_i$ with a $G_i$--translate, we can assume that there exists an element $a_i\in G_i$ such that $a_iU(\G_i)$ separates $U(\G_i)$ from $A\cap X_i$. This is a purely technical assumption to avoid an issue in the proof of Lemma~\ref{every g skewers} below.
\end{rmk}

Let us write $U_i=U(\G_i)$ for short and fix a shortest path $\beta_i\cu X_i$ from $U_i$ to $A\cap X_i$ 
(which is nonempty since $v_i\in\mc{A}$). 

Let $L_i\leq G_i$ denote the stabiliser of $U_i$ and let $L_i^0\leq L_i$ be the 
subgroup (of index at most two) that leaves invariant both connected components of $X_i- U_i$. Set $\overline 
U_i=L_i^0\backslash U_i$. This is a compact cube complex with a natural cubical immersion $\overline 
U_i\looparrowright\overline X_i'$ (recall that $\overline X_i'$ is the cubical subdivision of $\overline X_i$).

In fact, by construction, $U_i$ is CAT(0), because it is a tree of spaces whose vertex spaces are CAT(0) cube 
complexes and whose edge spaces are convex subcomplexes of the incident vertex spaces.  Since $L_i^0\acts U_i$ 
freely, we can identify $L_i^0$ with $\pi_1\overline U_i$, i.e.\ the immersion $\overline U_i\looparrowright\overline X_i$ 
induces the inclusion $L_i^0\hookrightarrow G_i$ at the level of fundamental groups.

\begin{figure}[h]
\begin{overpic}[width=0.66\textwidth,keepaspectratio]{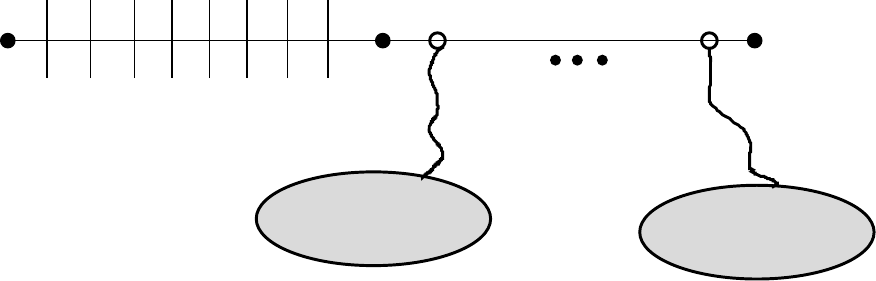}
\put(51,20){$\beta_1$}
\put(76,20){$\beta_k$}
\put(40,5){$\overline U_1$}
\put(85,5){$\overline U_k$}
\end{overpic}
\caption{The cube complex $\overline U\looparrowright G\backslash X$.}
\label{antenna2}
\end{figure}

We now assemble a cube complex $\overline U$ as in Figure~\ref{antenna2} by taking a copy of the tree 
$A\cup\beta_1\cup\dots\cup\beta_k\cu X$ and attaching a copy of $\overline U_i$ at the end of $\beta_i$ that does not lie on 
$A$. This comes equipped with a $\pi_1$--injective immersion $\overline U\looparrowright\overline 
X$.

The immersion $\overline U\looparrowright\overline X$ lifts to an embedding $U\hookrightarrow X$, where $U$ is the universal 
cover of $\overline U$; we also use the notation $U$ for the image of this embedding, which is shown in 
Figure~\ref{antenna3}.  (As shown in the figure, $U$ contains the tree $A\cup\beta_1\cup\dots\cup\beta_k\cu X$.)  
We identify the fundamental group $\pi_1\overline U$ with a 
subgroup $L\leq G$ that stabilises $U$ and acts cocompactly on it.  


\subsubsection{Quasiconvexity of $L$}\label{subsubsec:L_quasiconvex}
For each $v_i\in\mc{T}$, the intersection $U\cap X_i$ is the union of $U_i$ and all the $L_i^0$--translates of the path 
$\beta_i\cu X_i$. In particular, $U\cap X_i$ is at finite Hausdorff distance from $U_i$, hence quasiconvex.  For an 
arbitrary vertex $v\in\mc{T}$, the intersection $U\cap\pi^{-1}(v)$ is an $L$--translate of either some $U\cap X_i$, or some 
sub-path of the finite tree $A\cu X$. It follows that the intersections $U\cap\pi^{-1}(v)$ are uniformly quasiconvex 

Observe that $X$ decomposes as a tree of spaces with respect to the connected components of the various sets 
$\pi^{-1}(v)$, and $U$ is also a tree of spaces with respect to the components of $U\cap\pi^{-1}(v)$. It follows
from uniform quasiconvexity of the latter sets that $U$ is quasiconvex in $X$. Since the action $L\acts U$ is proper 
and cocompact, $L\leq G$ is a quasiconvex subgroup.

\begin{figure}[h]
\begin{overpic}[width=0.66\textwidth,keepaspectratio]{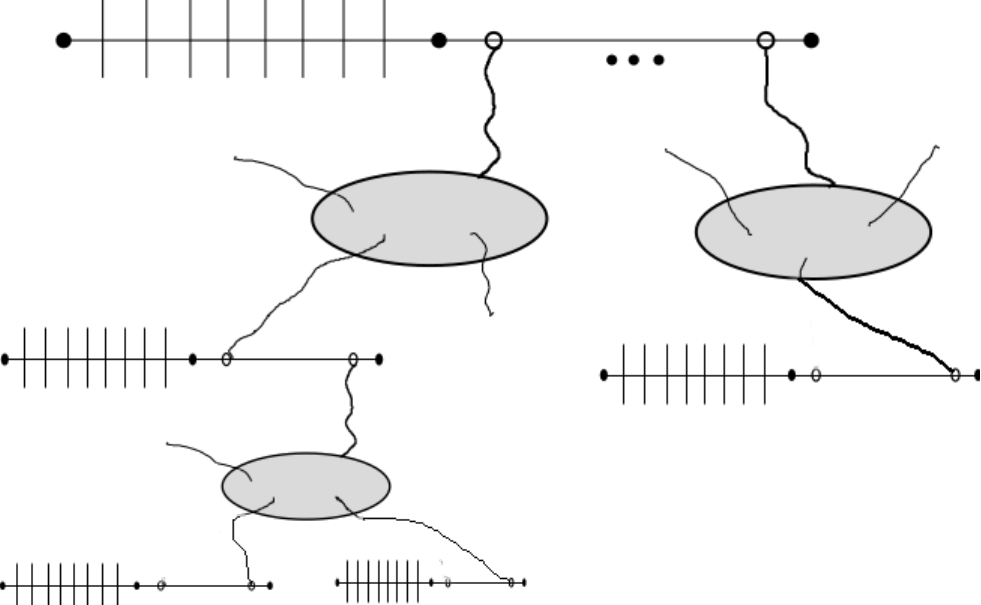}
     \put(43,37){$U_1$}
     \put(85,37){$U_k$}
     \put(52,51){$\beta_1$}
     \put(72,51){$\beta_k$}
\end{overpic}

\caption{The subcomplex $U\cu X$.}
\label{antenna3}
\end{figure}

\subsubsection{Cutting using $L$}\label{subsubsec:l_cut}
We now proceed to analyse the connected components of $X- U$.  Note that the projection 
$\mc{T}_U:=\pi(U)\cu\mc{T}$ is an $L$--invariant subtree and that the quotient $L\backslash\mc{T}_U$ is naturally identified 
with $\mc{A}$. We define an $L$--invariant map $\mf{p}\colon X\ra\mc{A}$ by composing the projection $\pi\colon X\ra\mc{T}$ 
with the nearest-point projection $\mc{T}\ra\mc{T}_U$ and, finally, the quotient projection $\mc{T}_U\ra 
L\backslash\mc{T}_U\cong\mc{A}$.

We also consider the $L$--invariant convex subset $\mc{U}=\pi^{-1}(\mc{T}_U)\cu X$, which contains $U$, and the 
$L$--equivariant gate-projection $p_{\mc{U}}\colon X\ra\mc{U}$. Observe that $\mf{p}\o p_{\mc{U}}=\mf{p}$. We denote by $C_i^+$ and $C_i^-$ the two connected components 
of $X_i- U_i$ and by $\mc{C}_i^{\pm}\cu\mc{U}$ the unions of all their $L$--translates.

\begin{lem}\label{labelling of components}
\begin{enumerate}
\item[]
\item For every vertex $w\in\mc{A}$, the set $\mf{p}^{-1}(w)- U$ is an $L$--invariant union of connected components 
of $X- U$.  Every connected component of $X- U$ is contained in one of these sets. 
\item  If $C$ is a connected component of $\mf{p}^{-1}(v_i)- U$, the projection $p_{\mc{U}}(C)$ is contained in 
either $\mc{C}_i^+$ or $\mc{C}_i^-$ (and this property is $L$--invariant).
\end{enumerate}
\end{lem}
\begin{proof}
Recall that the map $\mf{p}\colon X\ra\mc{A}$ is continuous and $L$--invariant.  If $x\in\mc{A}$ is a point in the interior 
of an edge, we have $\mf{p}^{-1}(x)\cu U$. Hence $X- U$ is a disjoint union of the finitely many, closed, 
$L$--invariant subsets $\mf{p}^{-1}(w)- U$, where $w\in\mc{A}$ is a vertex. This proves part~(1).
 
Recall that the $L$--stabiliser of $X_i\cu X$ is $L_i^0$.  Given that $L_i^0$ does not swap the two sides of $U_i\cu X_i$, 
the sets $\mc{C}_i^+$ and $\mc{C}_i^-$ are disjoint. Since $\mf{p}\o p_{\mc{U}}=\mf{p}$, a point $x\in X$ lies in 
$\mf{p}^{-1}(v_i)$ for some $1\leq i\leq k$ if and only if the projection $p_{\mc{U}}(x)$ lies in an $L$--translate of the 
subset $X_i\cu X$. In particular, $p_{\mc{U}}(\mf{p}^{-1}(v_i)- U)\cu\mc{C}_i^+\sqcup\mc{C}_i^-$. Observing that the 
$\mc{C}_i^{\pm}$ are open in the union $\mc{C}_i^+\sqcup\mc{C}_i^-$, we obtain part~(2).
\end{proof} 

Recall that each element $P\in\mscr{P}$ is a length-two sub-path $P\cu\mc{A}$; we denote by $z_P^{\pm}\in\mc{A}$ its two 
endpoints.  Let $\mc{H}^-\cu X$ be the union of all connected components of $X- U$ that are contained either in 
$\mf{p}^{-1}(z_P^-)$ for some $P\in\mscr{P}$, or in $p_{\mc{U}}^{-1}(\mc{C}_i^-)$ for some $1\leq i\leq k$. We also set 
$\mc{H}^+:=X-(\mc{H}^-\sqcup U)$. In particular, $\mc{H}^+$ contains all connected components of $X- U$ that 
are either contained in some $\mf{p}^{-1}(z_P^+)$ or in some $p_{\mc{U}}^{-1}(\mc{C}_i^+)$.

We obtain an $L$--invariant partition $X=\mc{H}^-\sqcup U\sqcup\mc{H}^+$.  By part~(3) of Lemma~\ref{connected components in general}, 
this gives rise to an abstract hyperplane $(L,\mf{H})$ (cf.\ Definition~\ref{abstract hyperplanes defn}), where 
$\mf{H}^{\pm}=\L(U\sqcup\mc{H}^{\pm})-\L U$. Note that the sets $\mf{H}^{\pm}$ are both nonempty as, for each $1\leq 
i\leq k$, the intersection $p_{\mc{U}}(U)\cap X_i$ is at finite Hausdorff distance from $U_i$ and 
$\partial_{\infty}C_i^{\pm}-\partial_{\infty}U_i\neq\emptyset$.

Now, applying Proposition~\ref{cubulating with abstract hyperplanes} to the collection of abstract hyperplanes 
$G\cdot\mf{H}$, we obtain a cocompact, essential $G$--action on a $\CAT$ cube complex with a single orbit of hyperplanes. In 
order to complete the proof of Theorem~\ref{one orbit intro}, we are only left to show that this action is proper. By 
part~(2) of Proposition~\ref{cubulating with abstract hyperplanes}, this amounts to the following:

\begin{lem}\label{every g skewers}
Every $g\in G-\{1\}$ has a conjugate $h$ with $h^+\in\mf{H}^+$ and $h^-\in\mf{H}^-$ (or $h^-\in\mf{H}^+$ and $h^+\in\mf{H}^-$).
\end{lem} 
\begin{proof}
There are two cases to consider, depending on whether $g$ lies in $G_{BS}$. 

If $g\in G_{BS}$, we can replace $g$ with a conjugate so that a path $P\in\mscr{P}$ is contained in its axis $\g\cu\mc{T}$.  
Any axis $\g'\cu X$ will satisfy $\pi(\g')=\g$ and $\pi(\g'\cap U)=P$. It 
follows that $\g'- U$ contains sub-rays lying in $\mc{H}^+$ and $\mc{H}^-$. Without loss of generality, we have
$g^+\in\L(U\sqcup\mc{H}^+)$ and $g^-\in\L(U\sqcup\mc{H}^-)$. Since no power of $g$ stabilises $U$, Lemma~\ref{intersection of 
limit sets} guarantees that $g^{\pm}\not\in\L U$. Thus $g^+\in\mf{H}^+$ and $g^-\in\mf{H}^-$, as required.

Suppose instead that $g\not\in G_{BS}$. Then, replacing $g$ with a conjugate, we can assume that $g$ admits an axis $\g\cu X$ that intersects one of the spaces $X_i\cu X$ in a non-trivial geodesic. By our choice of the crooked hyperplanes $\G_i$, the geodesic $\g\cap X_i$ intersects $U_i$ in a single point. Thus, either $\g\cu X_i$ and $C_i^{\pm}$ each contain a sub-ray of $\g$, or $\g\cap X_i$ is a finite segment with one endpoint in $C^+$ and one endpoint in $C^-$. 

We conclude that the sets $\mc{H}^{\pm}$ each contain a sub-ray of $\g$, unless possibly if $\g\cap X_i$ is a segment with one of its endpoints in $A\cap X_i$. This issue can be avoided simply by conjugating $g$ by the element $a_i\in G_i$ mentioned in Remark~\ref{annoying issue}. 

Without loss of generality, we have $g^+\in\L(U\sqcup\mc{H}^+)$ and $g^-\in\L(U\sqcup\mc{H}^-)$ once again. It is clear that no power of $g$ stabilises $U$. As before, we conclude that  $g^+\in\mf{H}^+$ and $g^-\in\mf{H}^-$.
\end{proof}

We have proved:

\begin{cor}\label{cor:one_orbit}
Every cocompactly cubulated hyperbolic group admits a cocompact, essential cubulation with a single orbit of hyperplanes.
\end{cor}

\section{The number of bald cubulations.}

\subsection{Bald cubulations of hyperbolic groups.}\label{sec:bald_hyperbolic}

Let $G$ be a cocompactly cubulated, non-elementary hyperbolic group. We first assume that $G$ is one-ended.

\begin{lem}\label{minimal hyperplane}
Suppose that $G$ admits only finitely many bald cubulations up to equivariant cubical isomorphism. Then there 
exists a bald cubulation $G\acts X$ 
and a hyperplane $\mf{w}\in\mscr{W}(X)$ with the following property. Let $G\acts Y$ be a bald cubulation.  
Then for each hyperplane $\mf u\in\mscr W(Y)$, $\L\mf u$ is not properly contained in $\L\mf w$.
\end{lem}
\begin{proof}
If the lemma did not hold, there would exist an infinite sequence of bald cubulations $G\acts X_n$ and hyperplanes $\mf{w}_n\in\mscr{W}(X_n)$ 
such that  $\L\mf{w}_{n+1}\subsetneq\L\mf{w}_n$. Since $G$ is virtually special, it has a finite-index torsion-free subgroup $H\leq G$. 
Given that $G$ has only finitely many bald cubulations, and each has only 
finitely many $H$--orbits of hyperplanes, there exist $h\in H$ and $m<n$ with 
$h\L\mf{w}_m=\L\mf{w}_n\subsetneq\L\mf{w}_m$. 

Observing that $h$ has infinite order, it follows that $h^+\in\L\mf{w}_m$. Hence, by Lemma~\ref{intersection of limit sets}, we have $\L(\langle h\rangle\cap{\rm Stab}_G(\mf{w}_m))=\{h^{\pm}\}\cap\L\mf{w}_m\neq\emptyset$. It follows that $\langle h\rangle\cap{\rm Stab}_G(\mf{w}_m)$ is infinite, i.e.\ a positive power of $h$ stabilises $\L\mf{w}_m$. This is a contradiction.
%
%
%
\end{proof}

Suppose $G$ has finitely many bald cubulations.  Let $G\acts X$ and $\mf{w}\in\mscr{W}(X)$ be the cubulation and hyperplane 
provided by Lemma~\ref{minimal hyperplane}. By essentiality of $G\acts X$, there exists an element $g\in G$ skewering 
$\mf{w}$. By \cite{Haglund}, we can replace $g$ with a power to ensure that it admits an axis $\g\cu X$. Let $p$ be the 
vertex of $\mf{w}$ that lies on $\g$. Let $K$ be the constant in Proposition~\ref{first properties} and define $M$ as at the 
beginning of Section~\ref{systems of switches sect}.

For every $n>2M$, let $\mscr{S}_n=(\mc{S}_n,H_n)$ be a full $n$--system of switches (its existence is guaranteed by 
Lemma~\ref{constructing systems of switches 2}). Replacing $\mscr{S}_n$ with 
\[k\cdot\mscr{S}_n:=(k\mc{S}_n,kH_nk^{-1})=(k\mc{S}_n,H_n)\] 
for some $k\in G$, we can assume that $\mf{w}\in\supp\mscr{S}_n$. Let $[p]_n\in\comp\mscr{S}_n$ denote the component that 
contains $p$.  Again replacing $\mscr{S}_n$ with $k\cdot\mscr{S}_n$ for some $k\in G_{\mf{w}}$, Lemma~\ref{large components} 
enables us to assume that $d(p,\mf{w}-[p]_n)\geq\tfrac{n}{2}-M$.

Now, Proposition~\ref{existence} guarantees the existence of a crooked hyperplane $\G_n\cu\mc{G}(\mscr{S}_n)$ that contains 
the vertex $[p]_n\in\mc{G}(\mscr{S}_n)$,  but not the entire subtree $\mc{G}(\mf{w})\cu\mc{G}(\mscr{S}_n)$. We will need the 
following observations.

\begin{lem}\label{required properties}
\begin{enumerate}
\item []
\item We have $\partial_{\infty}\mf{w}-\partial_{\infty}U(\G_n)\neq\emptyset$.
\item If $n>2(2K+M+\diam\pi_{\mf{w}}(\g))$, then $U(\G_n)$ is skewered by a power of $g$. 
\item For every open neighbourhood $\mc{V}$ of $\Lambda\mf{w}$ in $\partial_{\infty}G$,  there exists $\overline n$ such that 
$\Lambda U(\G_n)\cu\mc{V}$ for all $n\geq\overline n$.
\end{enumerate}
\end{lem}
\begin{proof}
Since $\mc{G}(\mf{w})\not\subseteq\G_n$, there exists a ray $r_n$ contained in the subtree 
$\mc{G}(\mf{w})\cu\mc{G}(\mscr{S}_n)$ and disjoint from $\G_n$.  As in Remark~\ref{boundaries of crooked hyperplanes}, the 
corresponding standard path $\tilde r_n\cu\mf{w}\cu X$ is a quasi-geodesic ray defining a point of 
$\partial_{\infty}\mf{w}-\partial_{\infty}U(\G_n)$. This shows part~(1). 

Regarding part~(2), note that $d(p,\mf{w}-[p]_n)\geq\tfrac{n}{2}-M>2K+\diam\pi_{\mf{w}}(\g)$. 
By Lemma~\ref{most stuff is far}, $p$ is the only point of intersection between $\g$ and $U(\G_n)$. 
As in the proof of Corollary~\ref{skewering cor}, we conclude that a power of $g$ skewers $U(\G_n)$.

Finally, Lemma~\ref{most stuff is far} yields $d(p,\pi_{\mf{w}}(U(\G_n)-\mf{w}))\geq\tfrac{n}{2}-M-2K$, which 
diverges for $n\ra+\infty$.  Recall that, by Remark~\ref{boundaries of crooked hyperplanes}, points of 
$\partial_{\infty}U(\G_n)$ are represented by uniform quasi-geodesic rays contained in $U(\G_n)$. It follows that the limit 
sets $\partial_{\infty}U(\G_n)$ Hausdorff-converge to $\partial_{\infty}\mf{w}$ with respect to the visual metric on 
$\partial_{\infty}X$ determined by the point $p$. This proves part~(3).
\end{proof}

We are now ready to prove Theorem~\ref{infinitely many intro}.

\begin{proof}[Proof of Theorem~\ref{infinitely many intro}]
Let us assume for a moment that the theorem has been proved in the one-ended case. If $G$ is virtually 
free, then the theorem follows from Lemma~\ref{lem:virtually_free} below. If $G$ is neither virtually free nor one-ended, 
$G$ has 
at least one one-ended factor in its maximal splitting over finite subgroups and the theorem follows from 
Lemma~\ref{assembling cubulations}. It remains to handle the one-ended case.

\textbf{The one-ended case:}  Suppose now that $G$ is one-ended. Let the cubulation $G\acts X$ and the crooked hyperplanes 
$\G_n$ be those constructed above.  We denote by $\mc{H}$ the collection of abstract hyperplanes arising from the hyperplanes 
of $X$, and by $\mf{H}_n$ the abstract hyperplane determined by $U(\G_n)$. 

Let $G\acts X_n$ be the essential cubulation arising from the set ${\mc{H}\cup G\cdot\mf{H}_n}$ via 
Proposition~\ref{cubulating with abstract hyperplanes}.  Let $G\acts (X_n)_{\bullet}$ be the bald cubulation provided by 
Proposition~\ref{facts about panel collapse}. Recall that, by part~(2) of Lemma~\ref{required properties}, the hyperplane of $X_n$ 
corresponding to $\mf{H}_n$ is skewered by a power of $g$ for all large $n$. Thus, part~(2) of 
Proposition~\ref{facts about panel collapse} guarantees that a power of $g$ 
also skewers a hyperplane $\mf{u}_n\in\mscr{W}((X_n)_{\bullet})$ with $\L\mf{u}_n\cu\L U(\G_n)$.

In each bald cubulation of $G$, only finitely many $\langle g\rangle$--orbits of hyperplanes are skewered by a power of $g$.  
If $G$ admitted only finitely many bald cubulations, infinitely many limit sets $\L\mf{u}_n$ would lie in the same $\langle g 
\rangle$--orbit. There are two cases to consider. Note that $\Lambda\mf{u}_n\neq\emptyset$ for all $n$, as $G$ is one-ended.

{\bf Case~1:} \emph{there exist two diverging sequences $(a_k)$ and $(b_k)$ with the property that 
$g^{b_k}\L\mf{u}_{a_0}=\L\mf{u}_{a_k}$.}  Since a power of $g$ skewers $\mf{u}_{a_0}$, the subsets $g^{b_k}\L\mf{u}_{a_0}$ 
Hausdorff-converge to $g^+$. This contradicts part~(3) of Lemma~\ref{required properties}, as $\L\mf{u}_{a_k}\cu\L 
U(\G_{a_k})$ and $g^+\not\in\L\mf{w}$.

{\bf Case~2:} \emph{there exists a diverging sequence $(a_k)$ such that $\L\mf{u}_{a_k}$ is constant.}  Call $\Delta$ this 
subset of $\partial_{\infty}G$. Note that $\Delta\cu\bigcap_k\L U(\G_{a_k})$, which is contained in $\L\mf{w}$ by part~(3) of 
Lemma~\ref{required properties}. By minimality of $\L\mf{w}$, we have $\Delta=\L\mf{w}$. This implies that $\L\mf{w}\cu\L 
U(\G_{a_k})$, contradicting part~(1) of Lemma~\ref{required properties}.
\end{proof}

\begin{lem}\label{lem:virtually_free}
Let $G$ be a group with $\Out(G)$ infinite, and suppose that $G$ admits a proper, cocompact action on a $\CAT$ 
cube complex.  Then $G$ admits infinitely many bald cubulations, no two of which are $G$--equivariantly 
isomorphic.
\end{lem}

\begin{proof}
By Proposition~\ref{facts about panel collapse}, $G$ admits a bald cubulation $\rho\colon G\to\Aut(X)$.  Each $\phi\in\Aut(G)$ 
defines an action $\rho\circ\phi\colon G\to\Aut(X)$, which is again a bald cubulation.  For simplicity, given $g\in G$ and $x\in X$, we write 
$gx$ to mean $\rho(g)(x)$.

Let $\{x_1,\dots,x_k\}$ contain exactly one vertex of $X$ from each $\rho(G)$--orbit.  Let 
$g_1,\dots,g_m\in G$ generate $G$. Consider the constants $s=\max_jd(x_1,x_j)$ and $r=\max_id(x_1,g_ix_1)$.

Let $\phi\in\Aut(G)$ and suppose that $\rho\circ\phi$ and $\rho$ are equivalent. By definition, there is 
$\iota\in\Aut(X)$ such that $\iota(hx)=\phi(h)\iota(x)$ for all $h\in G,x\in X$.

Choose $h\in G$ so that $\iota(x_1)=hx_j$ for some $j\leq k$.  Then, for each $i\leq m$, we have 
$\phi(g_i)hx_j=\phi(g_i)(\iota(x_1))=\iota(g_ix_1)$.  So $d(\phi(g_i)hx_j,hx_j)=d(g_ix_1,x_1)\leq r$, from which the 
triangle inequality gives $d(h^{-1}\phi(g_i)hx_1,x_1)\leq r+2s$ for all $i\leq m$.  Hence we can re-choose $\phi$ in its 
outer class so that each $\phi(g_i)$ displaces $x_1$ by at most a distance depending only on $\rho$ and the (fixed) 
generating set of $G$.  There are finitely many possible choices for each $\phi(g_i)$, and hence there are only finitely 
many $\Phi\in\Out(G)$ such that $\Phi$ has a representative $\phi\in\Aut(G)$ with $\rho$ and $\rho\circ\phi$ equivalent.

Thus, if there were only finitely many equivalence classes of actions $\rho\colon G\to\Aut(X)$, we would have that $\Out(G)$ is 
finite, a contradiction.
\end{proof}

\begin{rmk}
In order to deal with general virtually free groups in the proof of Theorem~\ref{infinitely many intro}, 
one might be tempted to behave as in 
Section~\ref{subsubsec:antenna}: work in a torsion-free finite-index subgroup, and then use 
the same idea as Lemma~\ref{supergroups} to cubulate the original group.  Unfortunately, this 
does not preserve hyperplane-essentiality.
\end{rmk}

\subsection{Groups with few bald cubulations.}\label{sec:unique_bald}

In this subsection, we prove Proposition~\ref{bald intro} (cf.\ Proposition~\ref{unique bald cubulation} below) and the
following result mentioned in the introduction.

\begin{prop}\label{bald Z^n}
Let $X$ be an essential, hyperplane-essential $\CAT$ cube complex endowed with a proper, cocompact action of $\Z^n$.
Then $X$ is isomorphic to the standard tiling of $\R^n$.
\end{prop}
\begin{proof}
By the Cubical Flat Torus theorem \cite[Theorem~3.6]{WW}, there exists an invariant convex subcomplex $Y\cu X$ 
that splits as product of quasi-lines $C_1,\dots,C_n$. Since the $\Z^n$--action is essential, we have $Y=X$. 
Since $X$ is essential and hyperplane-essential, so is each $C_j$. Every hyperplane of an essential quasi-line is bounded. 
Thus, since each $C_j$ is hyperplane-essential, it follows that $C_j\cong\R$. In conclusion, $X\cong\R^n$.
\end{proof}

Before proving Proposition~\ref{unique bald cubulation}, we need to obtain a couple of lemmas.

\begin{lem}\label{median sub-algebras}
Let $G$ act cocompactly on a bald cube complex $X$ with no $\R$--factors. If $A\cu X^{(0)}$ is a $G$--invariant, nonempty median subalgebra, then $A=X^{(0)}$.
\end{lem}
\begin{proof}
Since the action $G\acts X$ is essential, every halfspace of $X$ 
intersects $A$ nontrivially. We obtain a $G$--equivariant map $r_A\colon\mscr{H}(X)\ra\mscr{H}(A)$ that takes each halfspace 
of $X$ to its intersection with $A$. By Lemma~6.5 in \cite{Bow1}, this map is surjective. By Lemma~\ref{cocompact 
hyperplanes}, every halfspace of $X$ is at finite Hausdorff distance from its intersection with $A$. By definition, no two 
halfspaces of $X$ are at finite Hausdorff distance, so the fibres of $r_A$ are singletons. We conclude that $r_A$ is a 
bijection, hence $A=X^{(0)}$.
\end{proof}

\begin{lem}\label{automatically non-elementary}
Let $G$ be a group such that no finite-index subgroup of $G$ admits a nontrivial additive homomorphism to 
$\R$.
Let $G$ act cocompactly on a proper, 
unbounded $\CAT$ space $\mc{X}$. Then every $G$--orbit 
in the visual boundary $\partial_{\infty}\mc{X}$ is infinite.
\end{lem}

\begin{proof}
By our assumptions, the visual boundary $\partial_{\infty}\mc{X}$ is nonempty.  Suppose for the sake of 
contradiction that $G$ has a finite orbit in $\partial_\infty\mc{X}$, so a 
finite-index  subgroup $G_0\leq G$ fixes a point $\xi\in\partial_{\infty}\mc{X}$.

Let 
$b_{\xi}\colon\mc{X}\ra\R$ be any 
Busemann function determined by $\xi$. Given any $x\in \mc{X}$, the map $\phi\colon G_0\ra\R$ defined by 
$\phi(g)=b_{\xi}(gx)-b_{\xi}(x)$ is easily seen to be an additive homomorphism. By our assumption on $G$, the map $\phi$ must 
vanish identically. Hence $G_0$ leaves invariant each horosphere around $\xi$, contradicting cocompactness of 
$G_0\acts\mc{X}$.
\end{proof}

\begin{prop}\label{unique bald cubulation}
For $i=1,2$, let $T_i$ be locally finite trees with all vertices of degree $\geq 3$. Let $U_i\leq\Aut(T_i)$ be closed,  
locally primitive subgroups generated by edge stabilisers and satisfying Tits' independence property. Then, for any 
uniform lattice $\G\leq U_1\x U_2$ with dense projections to $U_1$ and $U_2$, 
the standard action $\G\acts T_1\x T_2$ is the only bald cubulation of $\G$.
\end{prop}
\begin{proof}
Let $\G\acts X$ be a bald cubulation.  Let ${X_1\x \dots\x X_k}$ be the de Rham decomposition\footnote{We 
immediately have $k=2$ by quasi-flat rank considerations, but this is not necessary in the proof.} of $X$ and let 
$\G_0\leq\G$ be a finite-index subgroup leaving each factor invariant. Each $X_j$ is a locally finite, bald cube complex 
endowed with a cocompact $\G_0$--action. 

Observe that $U_1$ and $U_2$ are simple groups by the argument in \cite{Tits} (see e.g.\ Theorem~3.3 in 
\cite{Caprace-DeMedts}).  Theorem~0.8 in \cite{Shalom00} thus implies that every additive homomorphism ${\G_0\ra\R}$ vanishes 
identically, and the same holds for any finite-index subgroup of $\Gamma_0$. Lemma~\ref{automatically 
non-elementary} then yields that there are no finite $\G_0$--orbits in the visual 
boundaries $\partial_{\infty}X_j$.

Again by simplicity, $U_i$ has no finite-index open subgroups, so the projection of $\G_0$ to $U_i$ is dense.  By Theorem~1.5 
in \cite{CFI}, each action $\G_0\acts X_j$ extends to a continuous action of some $U_{i_j}$ on a $\G_0$--invariant median 
subalgebra\footnote{The stronger statement appearing as Theorem~1.5 in \cite{CFI} is not true in general. See Section~4 of 
\cite{Fio3} for a discussion (in particular Theorem~4.4 and Example~4.7).} $A_j\cu X_j^{(0)}$. By Lemma~\ref{median 
sub-algebras}, we actually have $A_j=X_j^{(0)}$.

Observe that each $U_{i_j}\acts X_j$ is cocompact and essential, since so is the $\G_0$--action.  Thus, 
hyperplane-stabilisers are proper, open subgroups of $U_{i_j}$ and they act cocompactly on the respective hyperplanes by 
Lemma~\ref{cocompact hyperplanes}. Theorem~A in \cite{Caprace-DeMedts} shows that all hyperplane-stabilisers of $U_{i_j}\acts 
X_j$ are compact, which means that all hyperplanes of each $X_j$ are compact. Since $X_j$ is hyperplane-essential, it must be 
a tree. Finally, by Lemma~1.4.7\footnote{This appears as Lemma~3.7 in the unpublished version of the article available 
online.} in \cite{BMZ}, $X_j$ must be $U_{i_j}$--equivariantly isomorphic to $T_{i_j}$. We conclude that $k=2$ and that $X$ 
is $\G$--equivariantly isomorphic to $T_1\x T_2$.
\end{proof}

\bibliography{mybib}

\begin{thebibliography}{GMRS98}

\bibitem[Ago13]{Agol}
Ian Agol.
\newblock The virtual {H}aken conjecture.
\newblock {\em Doc. Math.}, 18:1045--1087, 2013.
\newblock With an appendix by Agol, Daniel Groves, and Jason Manning.

\bibitem[BF95]{BF-stable}
Mladen Bestvina and Mark Feighn.
\newblock Stable actions of groups on real trees.
\newblock {\em Invent. Math.}, 121(2):287--321, 1995.

\bibitem[BF19a]{BF2}
Jonas Beyrer and Elia Fioravanti.
\newblock Cross ratios and cubulations of hyperbolic groups.
\newblock {\em arXiv:1810.08087v2}, 2019.

\bibitem[BF19b]{BF1}
Jonas Beyrer and Elia Fioravanti.
\newblock Cross ratios on {CAT (0)} cube complexes and marked length-spectrum
  rigidity.
\newblock {\em arXiv preprint arXiv:1903.02447}, 2019.

\bibitem[BH99]{BH}
Martin~R. Bridson and Andr\'e Haefliger.
\newblock {\em Metric spaces of non-positive curvature}, volume 319 of {\em
  Grundlehren der Mathematischen Wissenschaften [Fundamental Principles of
  Mathematical Sciences]}.
\newblock Springer-Verlag, Berlin, 1999.

\bibitem[BHS17]{HHSI}
Jason Behrstock, Mark~F. Hagen, and Alessandro Sisto.
\newblock Hierarchically hyperbolic spaces, {I}: {C}urve complexes for cubical
  groups.
\newblock {\em Geom. Topol.}, 21(3):1731--1804, 2017.

\bibitem[BM00a]{BM1}
Marc Burger and Shahar Mozes.
\newblock Groups acting on trees: from local to global structure.
\newblock {\em Inst. Hautes \'{E}tudes Sci. Publ. Math.}, (92):113--150 (2001),
  2000.

\bibitem[BM00b]{BM2}
Marc Burger and Shahar Mozes.
\newblock Lattices in product of trees.
\newblock {\em Inst. Hautes \'{E}tudes Sci. Publ. Math.}, (92):151--194 (2001),
  2000.

\bibitem[BMZ09]{BMZ}
Marc Burger, Shahar Mozes, and Robert~J. Zimmer.
\newblock Linear representations and arithmeticity of lattices in products of
  trees.
\newblock In {\em Essays in geometric group theory}, volume~9 of {\em Ramanujan
  Math. Soc. Lect. Notes Ser.}, pages 1--25. Ramanujan Math. Soc., Mysore,
  2009.

\bibitem[Bou97]{Bourdon-GAFA}
Marc Bourdon.
\newblock Immeubles hyperboliques, dimension conforme et rigidit\'{e} de
  {M}ostow.
\newblock {\em Geom. Funct. Anal.}, 7(2):245--268, 1997.

\bibitem[Bow98]{Bow-JSJ}
Brian~H. Bowditch.
\newblock Cut points and canonical splittings of hyperbolic groups.
\newblock {\em Acta Math.}, 180(2):145--186, 1998.

\bibitem[Bow13]{Bow1}
Brian~H. Bowditch.
\newblock Coarse median spaces and groups.
\newblock {\em Pacific J. Math.}, 261(1):53--93, 2013.

\bibitem[BW12]{Bergeron-Wise}
Nicolas Bergeron and Daniel~T. Wise.
\newblock A boundary criterion for cubulation.
\newblock {\em Amer. J. Math.}, 134(3):843--859, 2012.

\bibitem[Cap19]{Caprace-BMW}
Pierre-Emmanuel Caprace.
\newblock Finite and infinite quotients of discrete and indiscrete groups.
\newblock In {\em Groups {S}t {A}ndrews 2017 in {B}irmingham}, volume 455 of
  {\em London Math. Soc. Lecture Note Ser.}, pages 16--69. Cambridge Univ.
  Press, Cambridge, 2019.

\bibitem[CDM11]{Caprace-DeMedts}
Pierre-Emmanuel Caprace and Tom De~Medts.
\newblock Simple locally compact groups acting on trees and their germs of
  automorphisms.
\newblock {\em Transform. Groups}, 16(2):375--411, 2011.

\bibitem[CFI16]{CFI}
Indira Chatterji, Talia Fern\'os, and Alessandra Iozzi.
\newblock The median class and superrigidity of actions on {$\rm CAT(0)$} cube
  complexes.
\newblock {\em J. Topol.}, 9(2):349--400, 2016.
\newblock With an appendix by Pierre-Emmanuel Caprace.

\bibitem[Che00]{Chepoi}
Victor Chepoi.
\newblock Graphs of some {${\rm CAT}(0)$} complexes.
\newblock {\em Adv. in Appl. Math.}, 24(2):125--179, 2000.

\bibitem[CS11]{CS}
Pierre-Emmanuel Caprace and Michah Sageev.
\newblock Rank rigidity for {CAT}(0) cube complexes.
\newblock {\em Geom. Funct. Anal.}, 21(4):851--891, 2011.

\bibitem[Duf12]{Dufour}
Guillaume Dufour.
\newblock {\em Cubulations de vari{\'e}t{\'e}s hyperboliques compactes}.
\newblock PhD thesis, Paris 11, 2012.

\bibitem[Dun79]{Dunwoody1}
Martin~J. Dunwoody.
\newblock Accessibility and groups of cohomological dimension one.
\newblock {\em Proc. London Math. Soc. (3)}, 38(2):193--215, 1979.

\bibitem[Dun85]{Dunwoody2}
Martin~J. Dunwoody.
\newblock The accessibility of finitely presented groups.
\newblock {\em Invent. Math.}, 81(3):449--457, 1985.

\bibitem[Fio17]{Fio3}
Elia Fioravanti.
\newblock Superrigidity of actions on finite rank median spaces.
\newblock {\em arXiv:1711.07737v1}, 2017.

\bibitem[For02]{Forester}
Max Forester.
\newblock Deformation and rigidity of simplicial group actions on trees.
\newblock {\em Geom. Topol.}, 6:219--267, 2002.

\bibitem[Gen16]{Genevois2}
Anthony Genevois.
\newblock Coning-off {${\rm CAT}(0)$} cube complexes.
\newblock {\em arXiv:1603.06513v1}, 2016.

\bibitem[GL07]{Guirardel-Levitt}
Vincent Guirardel and Gilbert Levitt.
\newblock Deformation spaces of trees.
\newblock {\em Groups Geom. Dyn.}, 1(2):135--181, 2007.

\bibitem[GMRS98]{GMRS}
Rita Gitik, Mahan Mitra, Eliyahu Rips, and Michah Sageev.
\newblock Widths of subgroups.
\newblock {\em Trans. Amer. Math. Soc.}, 350(1):321--329, 1998.

\bibitem[GPR12]{GPR}
Mauricio Gutierrez, Adam Piggott, and Kim Ruane.
\newblock On the automorphisms of a graph product of abelian groups.
\newblock {\em Groups Geom. Dyn.}, 6(1):125--153, 2012.

\bibitem[Gur07]{Guralnik}
Dan Guralnik.
\newblock Coarse decompositions of boundaries for {${\rm CAT}(0)$} groups.
\newblock {\em arXiv:math/0611006v2}, 2007.

\bibitem[Hag07]{Haglund}
Fr\'ed\'eric Haglund.
\newblock Isometries of {$\rm CAT(0)$} cube complexes are semi-simple.
\newblock {\em arXiv:0705.3386v1}, 2007.

\bibitem[Hag08]{Haglund-GD}
Fr\'{e}d\'{e}ric Haglund.
\newblock Finite index subgroups of graph products.
\newblock {\em Geom. Dedicata}, 135:167--209, 2008.

\bibitem[Hag20]{LargeFacing}
Mark Hagen.
\newblock Large facing tuples and a strengthened sector lemma.
\newblock {\em arXiv preprint arXiv:2005.09536}, 2020.

\bibitem[Hea21]{Healy}
Brendan~Burns Healy.
\newblock Acylindrical hyperbolicity of {${\rm Out}(W_n)$}.
\newblock {\em Topology Proc.}, 58:23--36, 2021.

\bibitem[HS18]{Hagen-Susse}
Mark~F. Hagen and Tim Susse.
\newblock On hierarchical hyperbolicity of cubical groups.
\newblock {\em arXiv:1609.01313v2}, 2018.

\bibitem[HT19]{Hagen-Touikan}
Mark~F. Hagen and Nicholas W.~M. Touikan.
\newblock Panel collapse and its applications.
\newblock {\em Groups, Geometry, and Dynamics}, 13(4):1285--1334, 2019.

\bibitem[HW08]{Haglund-Wise-GAFA}
Fr\'{e}d\'{e}ric Haglund and Daniel~T. Wise.
\newblock Special cube complexes.
\newblock {\em Geom. Funct. Anal.}, 17(5):1551--1620, 2008.

\bibitem[HW09]{Hruska-Wise-packing}
G.~Christopher Hruska and Daniel~T. Wise.
\newblock Packing subgroups in relatively hyperbolic groups.
\newblock {\em Geom. Topol.}, 13(4):1945--1988, 2009.

\bibitem[HW14]{Hruska-Wise-finiteness}
G.~C. Hruska and Daniel~T. Wise.
\newblock Finiteness properties of cubulated groups.
\newblock {\em Compos. Math.}, 150(3):453--506, 2014.

\bibitem[HW15a]{Hagen-Wise1}
Mark~F. Hagen and Daniel~T. Wise.
\newblock Cubulating hyperbolic free-by-cyclic groups: the general case.
\newblock {\em Geom. Funct. Anal.}, 25(1):134--179, 2015.

\bibitem[HW15b]{Hsu-Wise}
Tim Hsu and Daniel~T. Wise.
\newblock Cubulating malnormal amalgams.
\newblock {\em Invent. Math.}, 199(2):293--331, 2015.

\bibitem[HW16]{Hagen-Wise2}
Mark~F. Hagen and Daniel~T. Wise.
\newblock Cubulating hyperbolic free-by-cyclic groups: the irreducible case.
\newblock {\em Duke Math. J.}, 165(9):1753--1813, 2016.

\bibitem[KM12]{Kahn-Markovic}
Jeremy Kahn and Vladimir Markovic.
\newblock Immersing almost geodesic surfaces in a closed hyperbolic three
  manifold.
\newblock {\em Ann. of Math. (2)}, 175(3):1127--1190, 2012.

\bibitem[Lev05]{Levitt}
Gilbert Levitt.
\newblock Automorphisms of hyperbolic groups and graphs of groups.
\newblock {\em Geometriae Dedicata}, 114(1):49--70, 2005.

\bibitem[LS07]{LS}
Markus Lohrey and G\'{e}raud S\'{e}nizergues.
\newblock When is a graph product of groups virtually-free?
\newblock {\em Comm. Algebra}, 35(2):617--621, 2007.

\bibitem[Nib02]{Niblo-cut}
Graham~A. Niblo.
\newblock The singularity obstruction for group splittings.
\newblock {\em Topology Appl.}, 119(1):17--31, 2002.

\bibitem[NR03]{Niblo-Reeves}
Graham~A. Niblo and Lawrence~D. Reeves.
\newblock Coxeter groups act on {${\rm CAT}(0)$} cube complexes.
\newblock {\em J. Group Theory}, 6(3):399--413, 2003.

\bibitem[OW11]{Ollivier-Wise}
Yann Ollivier and Daniel~T. Wise.
\newblock Cubulating random groups at density less than {$1/6$}.
\newblock {\em Trans. Amer. Math. Soc.}, 363(9):4701--4733, 2011.

\bibitem[Pet97]{Pettet}
Martin Pettet.
\newblock Virtually free groups with finitely many outer automorphisms.
\newblock {\em Transactions of the American Mathematical Society},
  349(11):4565--4587, 1997.

\bibitem[Rol98]{Roller}
Martin~A. Roller.
\newblock Poc sets, median algebras and group actions. {A}n extended study of
  {D}unwoody's construction and {S}ageev's theorem.
\newblock Preprint, University of Southampton, 1998.

\bibitem[Sag95]{Sag95}
Michah Sageev.
\newblock Ends of group pairs and non-positively curved cube complexes.
\newblock {\em Proc. London Math. Soc. (3)}, 71(3):585--617, 1995.

\bibitem[Sag97]{Sag97}
Michah Sageev.
\newblock Codimension-{$1$} subgroups and splittings of groups.
\newblock {\em J. Algebra}, 189(2):377--389, 1997.

\bibitem[Sag14]{Sageev}
Michah Sageev.
\newblock {$\rm CAT(0)$} cube complexes and groups.
\newblock In {\em Geometric group theory}, volume~21 of {\em IAS/Park City
  Math. Ser.}, pages 7--54. Amer. Math. Soc., Providence, RI, 2014.

\bibitem[Sha00]{Shalom00}
Yehuda Shalom.
\newblock Rigidity of commensurators and irreducible lattices.
\newblock {\em Invent. Math.}, 141(1):1--54, 2000.

\bibitem[Thu79]{Thurston}
William~P Thurston.
\newblock {\em The geometry and topology of three-manifolds}.
\newblock Princeton University Princeton, NJ, 1979.

\bibitem[Tit70]{Tits}
Jacques Tits.
\newblock Sur le groupe des automorphismes d'un arbre.
\newblock In {\em Essays on topology and related topics ({M}\'{e}moires
  d\'{e}di\'{e}s \`a {G}eorges de {R}ham)}, pages 188--211. 1970.

\bibitem[Wis96]{Wise-thesis}
Daniel~T. Wise.
\newblock {\em Non-positively curved squared complexes: {A}periodic tilings and
  non-residually finite groups}.
\newblock ProQuest LLC, Ann Arbor, MI, 1996.
\newblock Thesis (Ph.D.)--Princeton University.

\bibitem[Wis04]{Wise-GAFA}
Daniel~T. Wise.
\newblock Cubulating small cancellation groups.
\newblock {\em Geom. Funct. Anal.}, 14(1):150--214, 2004.

\bibitem[Wis07]{Wise-trees}
Daniel~T. Wise.
\newblock Complete square complexes.
\newblock {\em Comment. Math. Helv.}, 82(4):683--724, 2007.

\bibitem[Wis11]{Wise-qch}
Daniel~T. Wise.
\newblock The structure of groups with a quasiconvex hierarchy.
\newblock {\em Preprint}, 2011.

\bibitem[Wis12a]{Wise-book}
Daniel~T. Wise.
\newblock {\em From riches to raags: 3-manifolds, right-angled {A}rtin groups,
  and cubical geometry}, volume 117 of {\em CBMS Regional Conference Series in
  Mathematics}.
\newblock Published for the Conference Board of the Mathematical Sciences,
  Washington, DC; by the American Mathematical Society, Providence, RI, 2012.

\bibitem[Wis12b]{Wise-antenna}
Daniel~T. Wise.
\newblock Recubulating free groups.
\newblock {\em Israel J. Math.}, 191(1):337--345, 2012.

\bibitem[WW17]{WW}
Daniel~T. Wise and Daniel~J. Woodhouse.
\newblock A cubical flat torus theorem and the bounded packing property.
\newblock {\em Israel J. Math.}, 217(1):263--281, 2017.

\end{thebibliography}
\bibliographystyle{alpha}

\end{document}